\documentclass[a4paper, 11pt, twoside, reqno]{amsart}
  \usepackage[a4paper,inner=2cm,outer=2cm,top=4cm,bottom=3.5cm]{geometry}
\usepackage{amsmath,amscd}
\usepackage{amssymb}
 \usepackage{amsthm}
\usepackage{comment}
\usepackage{mathrsfs}
\usepackage{graphicx, xcolor}
\usepackage{xcolor}
\usepackage{mathtools}
\usepackage{mathrsfs}
\usepackage[normalem]{ulem}
\usepackage[ocgcolorlinks, linkcolor=blue,citecolor=red]{hyperref}

\usepackage{bm}
\usepackage{bbm}
\usepackage{url}

\usepackage[utf8]{inputenc}
\usepackage{mathtools,amssymb}
\usepackage{esint}

    \usepackage{amssymb}
\usepackage{tikz}
\usepackage{dsfont}
\usepackage{relsize}
\usepackage{url}
\usepackage{xcolor}
\usepackage{graphicx}
\usepackage{mathrsfs}
\usepackage[shortlabels]{enumitem}
\usepackage{lineno}
\usepackage{amsmath}
\usepackage{enumitem}
\usepackage{amsthm} 
\usepackage{verbatim}
\usepackage{dsfont}
\usepackage{tikz}
\numberwithin{equation}{section}

\allowdisplaybreaks

\mathtoolsset{showonlyrefs}

\graphicspath{{images/}}
\theoremstyle{definition}
\newtheorem{step}{Step}

\newtheorem{theorem}{Theorem}[section]
\newtheorem{lemma}[theorem]{Lemma}

\newtheorem{proposition}[theorem]{Proposition}

\allowdisplaybreaks[0]

\title[ Semi-linear wave equation]{Reconstruction of potential and damping coefficients in a semi-linear wave equation}

\author[R. Bhardwaj]{Rahul Bhardwaj}
\author[M. Kumar]{Mandeep Kumar}
\author[M. Vashisth]{Manmohan Vashisth}
\address{{Department of Mathematics, Indian Institute of Technology Ropar, Rupnagar, Punjab-140001, INDIA.}}

\email{bhardwaj161067@gmail.com}
\email{mandeep.sansanwal@gmail.com}
\email{manmohanvashisth@iitrpr.ac.in}

\newcommand{\R}{{\mathbb R}}

\newcommand{\eps}{\epsilon}

\DeclareMathOperator{\supp}{supp} 

\begin{document}
	\begin{abstract}
In this article, we investigate an inverse problem for a semi-linear wave equation posed on bounded domain in $\mathbb{R}^{n+1}$, with $n \geq 2$. Our primary objective is to reconstruct the damping coefficient, the linear and nonlinear potentials from the associated Dirichlet-to-Neumann map. The analysis is based on a \emph{higher-order linearization} method. As a key step, we establish the existence of suitable asymptotic solutions, crucial for reconstructing the nonlinear potential. In addition, we also provide a detailed study of the corresponding forward problem.

\vspace{.2cm}
		

	\end{abstract}
	\maketitle
	 \tableofcontents
   \section{Introduction}
 \subsection{Mathematical setup and statement of main result}
 Let $\Omega$ be an open, connected and bounded subset of $\mathbb{R}^n$, $n\geq 2$, with a smooth boundary denoted by $\partial\Omega$. For $T>0$, we define the space-time domain $\Omega_T := (0,T)\times\Omega$ and denote its lateral boundary by $\Sigma := (0,T)\times\partial\Omega$. In the present article, we consider the following initial-boundary value problem (IBVP) for a semi-linear wave Equation with lower order perturbations 
  \begin{align}\label{equation; IBVP}
    \begin{cases}
      \Box u(t,x) + a(x)\partial_t u(t,x) + b(x) u(t,x) + q(x) u^{\ell}(t,x) = 0,  &(t,x) \in \Omega_T,\\
    u(t,x)  = f(t,x), & (t,x) \in\Sigma,\\
     u(0, x )=0,~~ \partial_t u(0, x ) = 0,  & \quad x \in\Omega,
    \end{cases}
\end{align}
where $\Box$ is the standard wave operator $\partial_t^2 - \Delta_x$ in $\mathbb{R}^{1+n}$, $a$ represent the damping coefficient,  $b$ and $q$  represent the linear and nonlinear potentials respectively,  and $\ell\ge2$ is an integer. 

Now for $a,b,q\in C_{c}^{\infty}(\Omega)$  and $f\in \mathcal{D}^{\delta}_{m+1}$ (see \eqref{E-delta} for its definition), 
we proved in Theorem \ref{main_thm:welposedness}, (see  section \ref{preliminaries}), that there exists a unique solution $u_f$ to \eqref{equation; IBVP} such that $\partial_\nu u_f\in H^{m}(\Sigma)$ and $\lVert \partial_\nu u_f\rVert_{H^{m}(\Sigma)}\leq C \lVert f\rVert_{H^{m+1}(\Sigma)}$ for some constant $C>0$ depending only on $a,b,q,\Omega$ and $T$. Here $\partial_{\nu} u_f|_{\Sigma}:=\nu\cdot\nabla u_f|_{\Sigma}$, with $\nu$ denotes the outward unit normal to  $\partial\Omega$. Building on this, we observe that the Dirichlet-to-Neumann (DN) map $\Lambda_{a,b,q}:\mathcal{D}^\delta_{m+1} \rightarrow H^{m}(\Sigma)$ given by 
\begin{align}\label{eq:DN map}
\begin{aligned}
\Lambda_{a,b,q}(f):=\partial_{\nu} u_f \big|_{\Sigma},\ \ f\in \mathcal{D}^\delta_{m+1}
\end{aligned} 
\end{align}
is well-defined whenever $u_f$ is solution to \eqref{equation; IBVP} corresponding to the   Dirichlet data $f\in \mathcal{D}^\delta_{m+1}$.

The main  objective of this article is to 
determine the 
coefficients $a$, $b$, and $q$ appearing in equation \eqref{equation; IBVP}
from the boundary measurements of solutions. 
We derive explicit reconstruction formulas that allow us to recover the coefficients  $a$, $b$, and $q$ from the knowledge of  DN map $\Lambda_{a,b,q}(f)$ measured for all $f\in \mathcal{D}^{\delta}_{m+1}$. 
More precisely, one of the main results of this article is stated in the following Theorem. 
\begin{theorem}\label{th: main result}
   {\it  Let $\Omega \subset \mathbb{R}^n$ be a bounded domain with smooth boundary $\partial\Omega$,  
 and let $T > \operatorname{diam}(\Omega)$.
Then the damping coefficient $a \in C_c^{\infty}(\Omega)$ and the potential terms $b,q \in C_c^{\infty}(\Omega)$  appearing in the semi-linear wave  Equation \eqref{equation; IBVP} 
are uniquely reconstructed from knowledge of the Dirichlet-to-Neumann map $\Lambda_{a,b,q}(f)$ measured for all $f \in \mathcal{D}^\delta_{m+1}$.}
\end{theorem}
This result can be viewed as a continuation of the reconstruction result established in \cite{Lassas2022UniquenessRA}, where Lassas \emph{et al.} investigated an IBVP for a semi-linear wave Equation in $\mathbb{R}^{n+1}$, $n \geq 1$,
involving a time-dependent nonlinear potential without linear 
potential and damping terms. Using the knowledge of the DN map, they
established the reconstruction result for the nonlinear potential. 
However, due to the finite speed of propagation of the wave operator, recovery of time-dependent coefficients in the entire space-time domain is not possible from the DN map; for more details, see~[\cite{Kian2017}, section 1]. Consequently, the authors of \cite{Lassas2022UniquenessRA}, reconstructed the time-dependent potential only in a compact subset of the space-time domain. 
In contrast, the present work considers 
 time-independent damping effects together with both linear and nonlinear potential terms.  We show that all unknown coefficients that appear in the IBVP  \eqref{equation; IBVP} can be uniquely reconstructed from the DN map throughout the entire domain $\Omega$. 
\subsection{Physical significance and motivation}
Wave phenomena are ubiquitous in nature and occur across a wide range of physical systems, including acoustics, electromagnetism, elasticity, seismology, and fluid dynamics. The mathematical model considered in this work describes the essential features of wave propagation in a medium that can absorb energy, vary in space, and respond nonlinearly to wave amplitude. 
Specifically, we study the scalar function $u$ governed by
\begin{align}\label{eq:main}
    \Box u(t,x) + a(x)\partial_t u(t,x) + b(x)u(t,x) + q(x)u^{\ell}(t,x) = 0,\quad \text{ in}\qquad (0,T)\times\Omega,
\end{align}
where the $\Box$ operator plays a central role in wave propagation phenomena, as it naturally arises in the study of hyperbolic partial differential equations (PDEs) and their physical principles, such as causality and the finite speed of wave propagation. In addition, the coefficient functions $a$, $b$, and $q$ describe the physical properties of the medium.

The term $a$ represents a damping effect arising from friction or material absorption.
The term $b$ acts as a linear potential, describing a restoring or refractive effect arising from the resistance of the medium to deformation.
The coefficient $q$ introduces a nonlinear potential, which allows the internal characteristics of the medium to change with position. This nonlinear term captures the amplitude-dependent behavior that arises in many physical systems. Equation \eqref{eq:main}  can be viewed as a generalized damped nonlinear \emph{Klein–Gordon equation} (see \cite{morawetz1968time}), which arises in quantum mechanics. These types of models have many applications. For example, in geophysics, it describes how seismic waves travel through the layers of the Earth, where stiffness and damping vary with depth. From the perspective of inverse problems, measurements of the boundary contain valuable information about the coefficients $a$, $b$, and $q$. Reconstructing these parameters from the observed data helps reveal the hidden internal structure of the medium. 

\subsection{Existing articles}
The inverse problem considered in the present article can be put into a subset of the Calder\'on type inverse problems for PDEs. 
Let us start by recalling some key developments in the study of such kind of  inverse problems for linear PDEs. In the elliptic setting, the foundational work of Calder\'{o}n~\cite{Calderon1980}, now known as the Calder\'{o}n inverse problem, aims to determine the electrical conductivity of a medium from boundary measurements. The question of uniquely recovering a time-independent scalar potential associated with linear hyperbolic PDEs from boundary measurements was first addressed by Bukhgeim and Klibanov in~\cite{Bukhgeim1981}.

Inspired by the pioneering ideas of construction of complex geometric optics solutions introduced by Sylvester and Uhlmann~\cite{Sylvester1987AGU} in their study of the Calder\'{o}n problem associated to conductivity and Schr\"odinger equation, Rakesh and Symes in \cite{Rakesh01011988} established a uniqueness result to recover the time-independent potential in hyperbolic PDEs using the measurements of the Neumann-to-Dirichlet map. Furthermore, using the DN map, Rakesh in \cite{rakesh1990reconstruction}  derived a reconstruction formula for recovering the time-independent potential appearing in the wave Equation. In \cite{Isakov1991AnIH}, Isakov investigated a uniqueness issue for simultaneous determination of potential and damping coefficients in the time-independent case for the wave Equation. In a general geometric setting, the uniqueness of inverse problems for wave Equations with time-independent coefficients was established by Eskin \cite{MR2235639,MR2441006}. We also refer to the book [\cite{isakov2006inverse}, chapter 8], in which Isakov mentioned various types of inverse problems for hyperbolic PDEs. Over the past few decades, several works have emerged on the recovery of various coefficients in the linear wave Equation through boundary measurements; see, for example,~\cite{RammSjostrand1991,KSV,Krishnan17082020,lasiecka1986non,Hu2017DeterminationOS,Kian2017,Kian2016RecoveryOT,Salazar2010DeterminationOT,Stefanov1989InverseSP,MR4343270,kumar2025h,doi:10.1137/23M1588676,LiuSaksalaYan2025} and the references therein.

 Most studies on inverse problems for nonlinear PDEs employ the so-called  \emph{higher-order linearization} method, originally introduced by Isakov~\cite{Isakov1993OnUI} in the study of inverse problems for nonlinear parabolic PDEs.

For nonlinear wave Equations, Nakamura et al.~\cite{NakamuraWatanabe2008,NakamuraWatanabeKaltenbacher2009,Nakamura2020InverseIB} established uniqueness and reconstruction results for both linear and nonlinear
coefficients in nonlinear wave Equations posed on Euclidean domains, using
boundary measurements map. In a related work, Lin et al.~\cite{lin2024determining} studied an inverse problem for a time-dependent semi-linear wave Equation in the absence of a damping term. They proved uniqueness results for the nonlinear term as well as for the source terms (initial displacement and initial velocity) using measurements of the DN map. More recently, Xiang et al.~\cite{Qiu2025UniquenessRF} investigated the simultaneous recovery of multiple unknown nonlinear coefficients together with the source term for the semi-linear wave Equation in the absence of a damping term, based on the measurement of the DN map.
In a general geometric setting, a remarkable development by Kurylev, Lassas and Uhlmann~\cite{Kurylev2018InversePF} revealed that the presence of nonlinearity in a wave Equation can be advantageous for the associated inverse problem. Using the technique of \emph{higher-order linearization method}, they proved that the source-to-solution map uniquely determines the global topology, differentiable structure, and conformal class of the metric on a global hyperbolic Lorentzian manifold of dimension $3+1$. 
Later, Hintz and Uhlmann~\cite{Hintz2021TheDM} extended these results and proved that the DN map allows the recovery of both the Lorentzian metric and the nonlinear coefficient. For more general classes of nonlinear operators, Oksanen et al.~\cite{Oksanen2020InversePF} established global uniqueness results to determine coefficients in nonlinear real principal-type Equations from two different
boundary determination methods. 

Furthermore, nonlinear inverse problems arising in various physical models such as the Westervelt, Jordan-Moore-Gibson-Thompson (JMGT), nonlinear elastic wave, and nonlinear progressive wave Equations have been extensively investigated in recent years; see, for example,~\cite{Uhlmann2023DeterminationOT,DEHOOP2019347,Uhlmann2021NonlinearUI,Fu2023InversePO,Li2023InversePF,Uhlmann2022AnIB,Wang2016InversePF,CLOP,FIKO,Kurylev2014InversePF,LUW,LTZ,Lin2024Wellposedness,liu2025partialdatainverseproblem}. Motivated by the studies mentioned above, the present work focuses on an inverse problem for a semi-linear wave equation that accounts for both damping and potential effects. We also refer to several works on inverse problems for other kinds of nonlinear PDEs that employ the \emph{higher-order linearization} technique; see,
for instance, \cite{CARSTEA2019121, Lai2024PartialDI, Kian2020PartialDI, CFKKU,Choulli2021,HarrachLin2023Simultaneous,KrupchykMaSahooSaloStAmant2025,LinLiuZhang2022,LassasOksanenSahooSaloTetlow2025} and the references therein. These studies collectively demonstrate that nonlinear effects can encode rich geometric and physical information, enabling the recovery of quantities and structures that remain inaccessible in purely linear inverse problems.

\subsection{Organization of the article}
The rest of this article is organized as follows. In Section~\ref{preliminaries}, we collect the necessary background material which is used throughout the article and establish the well-posedness of the IBVP \eqref{equation; IBVP}. 
In section~\ref{sec: Asymptotic solutions}, we will construct asymptotic solutions, which are crucial for the reconstruction of the nonlinear potential. Finally, in section~\ref{Sec: main result}, we give a detailed proof of the main result stated in Theorem~\ref{th: main result}.


\section{Well-posedness for the forward problem}\label{preliminaries}

This section is devoted to establishing the well-posedness of the IBVP \eqref{equation; IBVP}. First, we set some notations and provide the definitions required for our subsequent analysis. For  $1\le p\le\infty$, and $k\in\mathbb{N}$, denote by  $L^p(\Omega)$ and $W^{k,p}(\Omega)$ as  usual Lebesgue spaces and the Sobolev spaces of order $k$, equipped with the norms 
\begin{equation*}
\begin{aligned}
\|u\|_{L^p(\Omega)} := 
\begin{cases}
    \left(\int_\Omega |u(x)|^p\,dx\right)^{1/p},\ \mbox{if}\  1\leq p<\infty\\
\operatorname*{ess\,sup}_{x\in\Omega}|u(x)|,\ \  \mbox{if}\ \ p=\infty
\end{cases}
    \end{aligned}
    \end{equation*} 
and 
\begin{equation*}
\begin{aligned}
\|u\|_{W^{k,p}(\Omega)} := \begin{cases}
    \left(\sum_{|\alpha|\le k}\|\partial^\alpha u\|_{L^p(\Omega)}^{\,p}\right)^{1/p},\ \ \mbox{if}\ \ 1\leq p<\infty\\
    \max_{|\alpha|\le k}\|\partial^\alpha u\|_{L^\infty(\Omega)},\ \ \mbox{if}\ \ p=\infty, 
\end{cases}
    \end{aligned}
    \end{equation*} 
respectively, where for an $n-$multi-index $\alpha$, $\partial^\alpha u$ stands for $\alpha$th order  weak derivative of $u$ and   \[\operatorname*{ess\,sup}_{x\in\Omega}|u(x)|:=\inf\left\{m>0:\ \mu\left(\{x\in \Omega:\ \lvert u(x)\rvert>m\}\right)=0\right\},\] $\mbox{with $\mu$ stands for Lebesgue measure on $\mathbb{R}^n$}.$
Also if  $p=2$, then the space $W^{k,2}(\Omega)$ denoted by $H^{k}(\Omega)$ becomes a  Hilbert space equipped with the following inner product
\begin{align}\label{eq:inner-Hk}
  \langle u , v \rangle_{H^{k}(\Omega)}
  := \sum_{|\alpha|\le k} \int_{\Omega}
  \partial^{\alpha} u(x)\, \overline{\partial^{\alpha} v(x)} \, dx,
 \quad \text{for}\ u, v \in H^{k}(\Omega).
\end{align}
We also denote by $H^k_0(\Omega)$, the completion of compactly supported smooth functions on $\Omega$ (denoted by $C_c^{\infty}(\Omega)$), with respect to the $\lVert\cdot\rVert_{H^k(\Omega)}.$ Using a standard result related to trace theory for Sobolev spaces (see \cite{evans2022partial}), we have that the Hilbert space $H^k_0(\Omega)$ is given by 
\begin{align*}
     H^{k}_{0}(\Omega)
  := \left\{ f\in H^{k}(\Omega)\;:\;
            \partial^{\alpha} f=0 \quad \text{on}\ \partial\Omega
            \ \text{ for } |\alpha|\le k-1 \right\},\ k\in \mathbb{N}. 
\end{align*}
Next, we define the time-dependent Sobolev spaces. For a Banach space $(X,\lVert \cdot\rVert_{X})$, we start with defining  
\begin{align*}\begin{aligned}
L^p\left(0,T;X\right)
  &:=\left\{ f:[0,T]\to X \text{ strongly measurable }:\;
        \int_0^T \|f(t)\|^p_{X}\,dt <\infty \right\},\ \ 1\leq p<\infty \ \ \mbox{and}\\
        C([0,T];X)&:=\left\{f:[0,T]\to X: \text{ $f$ is continuous on $[0,T]$ }\right\}
        \end{aligned} 
\end{align*}
equipped with the norms
\begin{align*}
\begin{aligned}
     \|f\|_{L^p(0,T;X)} := \left(\int_0^T \|f(t)\|^p_X\,dt\right)^{1/p}\ \mbox{and}\ \ \|f\|_{C([0,T];X)}:=\max_{t\in [0,T]}\lVert f(t)\rVert_{X} 
\end{aligned}
\end{align*}
respectively. 
For an integer $k\ge0$, we define the time-dependent Sobolev space $H^{k}\left(0,T;X\right)$ by 
\begin{align*}
     H^{k}\left(0,T;X\right)
  := \left\{ f:[0,T]\to X \,:\, \partial_t^{j} f \in L^2(0,T;X)
        \text{ for all } j=0,1,\dots,k\right\} 
\end{align*}
and it  is a Banach space with respect to the  norm given by 
\begin{align*}
      \|f\|_{H^{k}(0,T;X)}
  := \left(\sum_{j=0}^{k} \|\partial_t^{j}f\|_{L^2(0,T;X)}^{2}\right)^{1/2}.
\end{align*} 
 Following \cite[Chapter~1, Theorem~2.3]{lions1972nonhomogeneous1}, the following  continuous embedding
\begin{align*}
H^{k}(\Omega_T)\hookrightarrow H^{k}\!\left(0,T;L^{2}(\Omega)\right)\cap L^{2}\!\left(0,T;H^{k}(\Omega)\right),\ 
\end{align*}
holds and an equivalent norm on $H^{k}(\Omega_T)$ is given by
\begin{align*}
    \|u\|_{H^{k}(\Omega_T)} 
    = \|u\|_{H^{k}\!\left(0,T;L^{2}(\Omega)\right)}
    + \|u\|_{L^{2}\!\left(0,T;H^{k}(\Omega)\right)} .
\end{align*} 
Moreover, whenever the domain is time-dependent, we denote
\begin{align*}
     H^{k}_{0}(\Omega_T)
  := \left\{ f\in H^{m}(\Omega_T)\;:\;
            \partial_t^j f\big|_{t=0} \in H_{0}^{k-j}(\Omega)
            \ \text{for } j=0,\dots,k-1 \right\},
\end{align*}
and the Sobolev space at the lateral boundary $\Sigma:=(0,T)\times \partial\Omega$,  is given by
\begin{align*}
     H^{k}(\Sigma) := H^{k}\left(0,T;L^{2}(\partial\Omega)\right) \cap L^{2}\left(0,T;H^{k}(\partial\Omega)\right).
 \end{align*}
We also define the space $H^{m}_{0}\left(0,T;H^{k}(\partial\Omega)\right)$ by 
\begin{align*}
     H^{m}_{0}\left(0,T;H^{k}(\partial\Omega)\right)
  := \{ f\in H^{m}(0,T;H^{k}(\partial\Omega)): \partial_t^{j} f\big|_{t=0}=0 \text{ in } H^{k}(\partial\Omega)
      \ \text{for } j=0,\dots,m-1\}.
\end{align*}
Next, for integers $m\ge k\ge 0$, define the space
\begin{align*}
     C^{k}\left([0,T];H^{m-k}(\Omega)\right)
  :=\left\{\,u:[0,T]\to H^{m-k}(\Omega)\;:\;
           \partial_t^j u\in C\left([0,T];H^{m-k}(\Omega)\right)\ \text{for } j=0,\dots,k\right\},
\end{align*}
\begin{align*}
     \|u\|_{C^{k}\left([0,T];H^{m-k}\right)}
  := \sum_{j=0}^{k}\ \sup_{t\in[0,T]}\,\|\partial_t^{\,j}u(t)\|_{H^{m-k}(\Omega)},
\end{align*} 
and the energy space $\mathcal{E}_m$  of order $m$  is defined by
\begin{align*}
    \mathcal{E}_m := \bigcap_{k=0}^m C^k\left([0,T];H^{m-k}(\Omega)\right).
\end{align*}
Now using a standard result on time-dependent Sobolev spaces (see for example \cite{evans2022partial}), we obtain  that the energy space $\mathcal{E}_m$ of order $m$ is a 
Banach space with respect to the norm $\lVert \cdot\rVert_{\mathcal{E}_m}$, given by 
\begin{align}\label{norm-Em}
    \|u\|_{\mathcal{E}_m}^2 
    := \sup_{0 \le t \le T} 
    \sum_{k=0}^{m} 
    \bigl\| \partial_t^{\,k} u(t) \bigr\|_{H^{\,m-k}(\Omega)}^{2}.
\end{align}
Moreover, for $m>n+1$, the Sobolev embedding implies that $\mathcal{E}_m$ becomes an algebra; see \cite[Definition~3.5]{choquet2009general}. Moreover, it is continuously multiplicative in the sense that there exists $C_m>0$ (constant depending on $m$ only) such that
\begin{align}\label{eq:Banach_algebra}
\|\Phi\Psi\|_{\mathcal{E}_m}\le C_m\|\Phi\|_{\mathcal{E}_m}\,\|\Psi\|_{\mathcal{E}_m},\quad \text{for all }\,\,\Phi,\Psi\in \mathcal{E}_m.
\end{align} 
 Furthermore, for a positive integer $m$, we define the boundary trace class space $\mathcal{K}_{m+1}$, compatible with $H^{m}$ on $\Sigma$ by 
\begin{align}\label{eq; trace class}
    \mathcal{K}_{m+1} = \left\{ f\in H^{m+1}(\Sigma)\;:\; f\in H_0^{m+1-k}(0,T; H^k(\partial\Omega)), \, \text{for} \, \, k= 0,1,\dots, m  \right\},
\end{align}
and for a  $\delta>0$, sufficiently small, we define a set  $ \mathcal{D}^\delta_{m+1}$ by 
\begin{align}\label{E-delta}
    \mathcal{D}^\delta_{m+1} := \left \{ f \in \mathcal{K}_{m+1} \;:\; \|f\|_{H^{m+1}(\Sigma)}< \delta\right\}.
\end{align}
\subsection{Well-posedness of the semi-linear wave Equation}\label{well-posedness}
In this subsection, we establish the well-posedness of the semi-linear wave Equation described by IBVP \eqref{equation; IBVP}. We start with recalling a well-posedness result for an  IBVP for an  inhomogeneous linear wave Equation with potential given by 
\begin{align}\label{eq:wave-b}
\begin{cases}
\Box v(t,x) + b(x)\,v(t,x) = F(t,x), & (t,x)\in \Omega_T,\\
v(t,x) = f(t,x), & (t,x)\in \Sigma,\\
v(0,x)=\varphi(x),\quad \partial_t v(0,x)=\psi(x), & x\in \Omega.
\end{cases}
\end{align}
\begin{lemma}[Well-posedness for linear Equations with potential (\cite{lin2024determining} Lemma 3.2)]\label{lemma:potential}
{\it Let $m>n+1$ and $T>0$. Now if  $\varphi\in H^{m+1}_0(\Omega)$, $\psi\in H^{m}_0(\Omega)$, $f \in \mathcal{K}_{m+1}$,  $b\in C_c^{\infty}(\Omega)$, and $F\in \mathcal{E}_m$ with $\partial_t^{\,k} F(0)\in H^{m-k}_0(\Omega)$ for $k=0,1,\dots,m-2$, then IBVP \eqref{eq:wave-b} admits a unique solution $\displaystyle  v\in \mathcal{E}_{m+1}$ such that  $\displaystyle \partial_\nu v\in H^{m}(\Sigma)$ and 
\begin{align}\label{energy estimate-v}
\begin{aligned}
   &\|v\|_{\mathcal{E}_{m+1}} + \|\partial_\nu v\|_{H^m(\Sigma)}\\
   &\quad \le C\,e^{CT}\left(\sum_{k=0}^{m}\|\partial_t^{\,k} F\|_{L^1(0,T;H^{m-k}(\Omega))}
+ \|\varphi\|_{H^{m+1}(\Omega)} + \|\psi\|_{H^{m}(\Omega)} + \|f\|_{H^{m+1}(\Sigma)}\right)
\end{aligned}
\end{align}
where $C>0$,  depends only on   $\Omega,$ and $b$.} 
\end{lemma} 
Next, in   Lemma \ref{lem:potential and damping}, stated below,  we establish a well-posedness result for an  IBVP related to  a   linear wave Equation, given by 
\begin{align}\label{eq:wave-ab}
\begin{cases}
\Box v(t,x) + a(x) \,\partial_t v(t,x) + b(x)\,v(t,x) = F _1(t,x),& (t,x)\in \Omega_T,\\
v(t,x) = f(t,x), & (t,x)\in \Sigma,\\
v(0, x )=0,\quad \partial_{t}v(0, x )=0, & \quad x\in \Omega. 
\end{cases}
\end{align}
\begin{lemma}[Well-posedness for linear Equations with potential and damping terms]\label{lem:potential and damping}
{\it Let $m>n+1$ and $T>0$. Assume that $f \in \mathcal{K}_{m+1}$,  $a,b\in C_c^{\infty}(\Omega)$, and $F_{1}\in \mathcal{E}_{m}$ with  $\partial_t^{\,k} F_{1}(0)\in H^{m-k}_0(\Omega)$ for $k=0,1,\dots,m-1$, then the IBVP \eqref{eq:wave-ab}
admits a unique solution  
$\displaystyle 
    v\in \mathcal{E}_{m+1}$ such that   $\displaystyle \partial_\nu v\in H^{m}(\Sigma)$ 
and satisfies the following estimate 
\begin{align}\label{estimate;potential and damping terms}
   \|v\|_{\mathcal{E}_{m+1}} + \|\partial_\nu v\|_{H^{m}(\Sigma)}
\le C\,e^{CT}\left(\sum_{k=0}^{m}\|\partial_t^{\,k} F_1\|_{L^1(0,T;H^{m-k}(\Omega))}
+  \|f\|_{H^{m+1}(\Sigma)}\right)
\end{align} 
for some constant  $C>0$, depending  only on   $\Omega,a$ and $b$.}
\end{lemma}
\begin{proof}
Using the Banach fixed-point Theorem, we first prove the local existence of a solution to the IBVP \eqref{eq:wave-ab}. Next, in order to obtain the global solution, we use the structure of the Equation to iteratively extend the interval of local existence to the entire time interval. Hence, in the subsequent analysis, we divide the proof into two steps. In the first step, we prove local existence, and in the next step, we prove global existence from the local existence.

\begin{step} \label{step:1}
We first introduce the space $\mathcal{F}_{[0,T]}\subset \mathcal{E}_{m+1}$ by 
\[\mathcal{F}_{[0,T]}:=\left\{\, v\in \mathcal{E}_{m+1} \;:\; \partial_t^k v(0)\in H^{m+1-k}_0(\Omega)\ \text{ for } k=0,1,\dots,m \,\right\}\]  and equip  with the norm given by  
\begin{equation}\label{norm-k}
    \|v\|_{\mathcal{F}_{[0,T]}}:= \|v\|_{\mathcal{E}_{m+1}}+ \sum_{k=0}^{m} \left\|\partial_t^k v\right\|_{C([0,T];H^{m+1-k}(\Omega))}. 
\end{equation}
Then it can be shown  that the space  $\mathcal{F}_{[0,T]}$ defined above is a closed subspace of $\mathcal{E}_{m+1}$ and the   norms   $\lVert \cdot\rVert_{\mathcal{F}_{[0,T]}}$ defined in \eqref{norm-k} and the norm $\lVert \cdot \rVert_{\mathcal{E}_{m+1}}$  on $\mathcal{E}_{m+1}$ defined in \eqref{norm-Em}  are equivalent on  $\mathcal{F}_{[0,T]}$, hence in view of this,  we  use the norm $\lVert\cdot\rVert_{\mathcal{E}_{m+1}}$ on $\mathcal{F}_{[0,T]}$ instead of the norm given in  Equation \eqref{norm-k}. Now for $v\in \mathcal{F}_{[0,T]}$, we have $F_1-a\, \partial_t v \in \mathcal{E}_{m}$ and 
$\partial_t^k\left(F_1-a\, \partial_t v\right)(0)\in H^{m-k}_0(\Omega)$ whenever  $k=0,1, \dots ,m-1$, therefore using the  Lemma \ref{lemma:potential}, we obtain that  the following IBVP
 \begin{align}\label{eq:w}
\begin{cases}
\Box w(t,x) + b(x)\,w(t,x)= F_1(t,x) - a(x)\,\partial_tv(t,x), &(t,x)\in \Omega_T,\\
w(t,x) = f(t,x), &(t,x)\in \Sigma,\\
w(0, x )=0,\quad \partial_{t}w(0, x )=0, &\quad x \in \Omega,
\end{cases}
\end{align}
has a unique solution $
  w\in \mathcal{E}_{m+1}$ such that $ \partial_\nu w\in H^{m}(\Sigma)$ and  it satisfies the following estimate 
\begin{align}\label{eq:est-w}
\quad\|w\|_{\mathcal{E}_{m+1}}+\|\partial_\nu w\|_{H^{m}(\Sigma)}
\leq\;
C e^{CT}\left(
\sum_{k=0}^{m} \big\|\partial_t^k(F_1-a \,\partial_t v)\big\|_{L^1(0,T;H^{m-k}(\Omega))}
+\|f\|_{H^{m+1}(\Sigma)}
\right).
\end{align}
Next, we
define the operator 
\begin{align}\label{map:L}
    \mathcal{T}:\mathcal{F}_{[0,T]}\to \mathcal{E}_{m},\qquad \mathcal{T}(v):=w,
\end{align}
where $w$ solves the IBVP \eqref{eq:w} corresponding to $v\in \mathcal{F}_{[0,T]}$. By Lemma \ref{lemma:potential}, the map \eqref{map:L} is well-defined. Now since $\partial_t^k\left(F_1-a\, \partial_t v\right)(0)\in H^{m-k}_0(\Omega)$ for $k=0,1, \dots ,m-1$, hence the compatibility conditions are satisfied, therefore using Lemma \ref{lemma:potential}, we get that $w\in \mathcal{F}_{[0,T]}$. Hence,  the  map $\mathcal{T}$ can be redefined as 
\begin{align}\label{map:L-eq}
\mathcal{T}:\mathcal{F}_{[0,T]}\to\mathcal{F}_{[0,T]},\qquad \mathcal{T}(v)=w
\end{align}
where $w$ solves the IBVP \eqref{eq:w} corresponding to $v\in \mathcal{F}_{[0,T]}$. Now, in order to prove the local  existence of a solution to the IBVP \eqref{eq:wave-ab}, it is enough to 
show that there exists a $T_{0}>0$, such that  the map $\mathcal{T}$ defined by \eqref{map:L-eq} is a contraction map from $\mathcal{F}_{[0,T_0]}$ to itself and hence it has a fixed point in $\mathcal{F}_{[0,T_0]}$, this gives the local existence of solution to the IBVP \eqref{eq:wave-ab}. 
To prove that the map $\mathcal{T}$ is a contraction, we start with defining  $\mathcal{T}(v_i):=w_i$, for $i=1,2$ for  $v_1,v_2\in \mathcal{F}_{[0,T]}$. 
Now,  observe that $w_{1}-w_{2}$ solves the following IBVP 
\begin{align}\label{eq:w1-w2}
\begin{cases}
\Box \left(w_{1}-w_{2}\right)(t,x) + b(x)\,\left(w_{1}-w_{2}\right)(t,x)=  - a(x)\,\partial_t(v_{1}-v_{2})(t,x), &(t,x)\in \Omega_T,\\
\left(w_{1}-w_{2}\right)(t,x) = 0, &(t,x)\in \Sigma,\\
\left(w_{1}-w_{2}\right)(0, x )=0,\quad \partial_{t}\left(w_{1}-w_{2}\right)(0, x )=0, &\quad x \in \Omega.
\end{cases}
\end{align}
Using the estimate from Equation \eqref{eq:est-w}, we arrive at the following energy estimates for $w_{1}-w_{2}$
\begin{align}\label{eq:estw1-w2}
\quad\|w_{1}-w_{2}\|_{\mathcal{E}_{m+1}}
\leq\;
C e^{CT}
\sum_{k=0}^{m} \big\|\partial_t^k(-a \,\partial_t (v_{1}-v_{2}))\big\|_{L^1(0,T;H^{m-k}(\Omega))}. 
\end{align}
This after using the fact that $\mathcal{T}(v_i):=w_i,$ for $i=1,2$, gives us
\begin{align}
    \|\mathcal{T}(v_{1})-\mathcal{T}(v_{2}) \|_{\mathcal{E}_{m+1}} &\leq\; 
C e^{CT}
\sum_{k=0}^{m} \left\|\partial_{t}^{k}\left( a \,\partial_t (v_{1}-v_{2})\right)\right\|_{L^1(0,T;H^{m-k}(\Omega))}\\
&\leq \ CC_{m}Te^{CT} \lVert a\rVert_{H^{m}(\Omega)}\sum_{k=0}^{m}\left\| \,\partial_t^{k+1} (v_{1}-v_{2})\right\|_{C([0,T];H^{m-k}(\Omega))}\\
& \leq CC_{m}T e^{CT}\|a\|_{H^{m}(\Omega)}\|v_{1}-v_{2}\|_{\mathcal{E}_{m+1}}, 
\end{align}
where we have used the  fact that $\mathcal{E}_{m}$ is an algebra (see \eqref{eq:Banach_algebra}) and the estimate 
$\|f\|_{L^1(0,T;H^{p}(\Omega))}\le T\,\|f\|_{C([0,T];H^{p}(\Omega))}$, which holds for any $f\in L^1(0,T;H^{p}(\Omega))$ and  $p\in \mathbb{N}\cup\{0\}$. Now for $\mathcal{T}$ to be a contraction, we must choose  $T:=T_0$ such that $CC_{m}T_0 e^{CT_0}\|a\|_{H^{m}(\Omega)}<1$. Thus, we obtain that $\mathcal{T}|_{\mathcal{F}_{[0,T_0]}}:=\mathcal{T}_{T_0}$ has a fixed point in $\mathcal{F}_{[0,T_0]}$ where $\mathcal{F}_{[0,T_0]}$ is defined by 
\begin{align*}
    \mathcal{F}_{[0,T_{0}]}=\left\{\, v\in \mathcal{E}_{{m+1},T_{0}}:=  \bigcap_{k=0}^{m+1} C^k\left([0,T_{0}];H^{m+1-k}(\Omega)\right) \;:\; \partial_t^k v(0)\in H^{m+1-k}_0(\Omega)\ \text{ for } k=0,\dots,m \,\right\},
\end{align*}
and it is equipped with a norm given by 
\begin{equation}\label{norm-k-T0}
    \|u\|_{\mathcal{F}_{[0,T_{0}]}}^{2}:= \sup_{0 \le t \le T_{0}} 
    \sum_{k=0}^{m+1} 
    \bigl\| \partial_t^{\,k} u(t) \bigr\|_{H^{\,m+1-k}(\Omega)}^{2}.
    \end{equation}
Now using the fact that  $\mathcal{F}_{[0,T_{0}]}$ is a closed subspace of Banach space $\mathcal{E}_{m+1,T_{0}}$, we  have that  $\mathcal{F}_{[0,T_{0}]}$ is also a Banach space. Hence, using the Banach fixed-point Theorem, we obtain a unique $v\in \mathcal{F}_{[0,T_{0}]}$ such that $\mathcal{T}_{T_{0}}(v)=v.$ Thus, we obtain that for any $F_{1}\in \mathcal{E}_{m}$ with $\partial_t^{\,k} F_{1}(0)\in H^{m-k}_0(\Omega)$, $k=0,1,\dots,m-1$, there exists a unique solution $v\in \mathcal{F}_{[0,T_{0}]}$ to IBVP \eqref{eq:wave-ab}. Moreover, using the estimate from \eqref{eq:est-w}, we have that
\begin{align}
\|v\|_{\mathcal{F}_{[0,T_{0}]}}+\|\partial_\nu v\|_{H^{m}((0,T_{0})\times \partial \Omega)}
\leq\;
C e^{CT_{0}}\left(
\sum_{k=0}^{m} \big\|\partial_t^k(F_1-a \,\partial_t v)\big\|_{L^1(0,T_{0};H^{m-k}(\Omega))}
+\|f\|_{H^{m+1}((0,T_{0})\times \partial \Omega)}
\right)
\end{align}
holds for some constant $C>0$, depending only on $\Omega, a$ and $b$. 
Finally,  using the triangle inequality, we arrive at 
\begin{align}\label{}
&\|v\|_{\mathcal{F}_{[0,T_{0}]}}+\|\partial_\nu v\|_{H^{m}((0,T_{0})\times \partial \Omega)}\\
&\qquad\leq\;
C e^{CT_{0}}\left(
\sum_{k=0}^{m} \big\|\partial_t^k\!F_1 \big\|_{L^1(0,T_{0};H^{m-k}(\Omega))}
+ T_{0}C_{m}\|a\|_{H^{m}(\Omega)}\|v\|_{\mathcal{F}_{[0,T_{0}]}}
+\|f\|_{H^{m+1}((0,T_{0})\times \partial \Omega)}
\right)\\
& \qquad=\;
C e^{CT_{0}}\left(
\sum_{k=0}^{m} \big\|\partial_t^k\!F_1 \big\|_{L^1(0,T_{0};H^{m-k}(\Omega))}
+\|f\|_{H^{m+1}((0,T_{0})\times \partial \Omega)}
\right) + CC_{m}T_{0} e^{CT_{0}}\|a\|_{H^{m}(\Omega)}\|v\|_{\mathcal{F}_{[0,T_{0}]}}\\
& \qquad\leq\;
C e^{CT_{0}}\left(
\sum_{k=0}^{m} \big\|\partial_t^k\!F_1 \big\|_{L^1(0,T_{0};H^{m-k}(\Omega))}
+\|f\|_{H^{m+1}((0,T_{0})\times \partial \Omega)}
\right) \qquad \text{(By the choice of $T_{0}$)}
\end{align} 
for some constant $C>0$, depending only on $a,b$ and $\Omega$. This completes the local existence of the solution to the IBVP \eqref{eq:wave-ab}. In the next step, we establish the global existence of a solution to the IBVP \eqref{eq:wave-ab}.
\end{step} 
\begin{step}\label{step:2} 
    If $T\leq T_{0}$, then there is nothing to prove as the local solution becomes a global solution. Hence, without loss of generality, we can assume that $T_{0} < T$. From Step \ref{step:1}, we have the existence of a solution in $(0,T_{0})\times \Omega$. Next, following the approach used in Step \ref{step:1},  we  extend the solution to the IBVP \eqref{eq:wave-ab} from  $(0,T_0)\times \Omega$ to $ \left(\frac{T_{0}}{2},\frac{3T_{0}}{2}\right)\times \Omega$.
We begin by defining 
\begin{align*}
    \mathcal{F}_{\left[\frac{T_{0}}{2},T\right]}:=\left\{\, v\in  \bigcap_{k=0}^{m+1} C^k\left(\left[\frac{T_{0}}{2},T\right];H^{m+1-k}(\Omega)\right) \;:\; \partial_t^k v\left(\frac{T_{0}}{2} \right)\in H^{m+1-k}_0(\Omega)\ \text{ for } k=0,\dots,m \,\right\}
\end{align*}
with the norm
\begin{equation}\label{norm-k-3T0}
    \|v\|_{\mathcal{F}_{\left[\frac{T_{0}}{2},T\right]}}:=   \left(\sup_{\frac{T_{0}}{2} \le t \le T} 
    \sum_{k=0}^{m+1} 
    \bigl\| \partial_t^{\,k} v(t) \bigr\|_{H^{\,m+1-k}(\Omega)}^{2}\right)^{\frac{1}{2}}. 
\end{equation}
Let $v_{\mathrm{cont.}}\in \mathcal{F}_{\left[\frac{T_{0}}{2},T\right]}$, $\partial_t^{\,k} F_{1}\left(\frac{T_{0}}{2}\right)\in H^{m-k}_0(\Omega)$ for $k=0,1,\dots,m-1$ and $w_{\mathrm{cont.}}$ solve the following IBVP
\begin{align}\label{well-posedness-extw}
\begin{cases}
\partial_t^{2} w_{\mathrm{cont.}} - \Delta w_{\mathrm{cont.}} + b(x) w_{\mathrm{cont.}}
= F_{1}(t,x) - a(x)\partial_t v_{\mathrm{cont.}}(t,x),
& (t,x)\in\left(\frac{T_{0}}{2},T\right)\times\Omega,\\[1ex]
w_{\mathrm{cont.}}(t,x)=f(t,x),
& (t,x)\in\left(\frac{T_{0}}{2},T\right)\times\partial\Omega,\\[1ex]
w_{\mathrm{cont.}}\!\left(\tfrac{T_{0}}{2},x\right)=w\!\left(\tfrac{T_{0}}{2},x\right),\ \ 
\partial_t w_{\mathrm{cont.}}\!\left(\tfrac{T_{0}}{2},x\right)
=\partial_t w\!\left(\tfrac{T_{0}}{2},x\right)
& x\in\Omega .
\end{cases}
\end{align}
where $w$ is the unique solution to the IBVP \eqref{eq:wave-ab} on $(0,T_0)\times \Omega$, established in Step \ref{step:1}. 
 Now using Lemma \ref{lemma:potential}, IBVP \eqref{well-posedness-extw} admits a unique solution. Moreover, it satisfies the following estimate 
\begin{align}
&\|w_{\mathrm{cont.}}\|_{\mathcal{F}_{\left[\frac{T_{0}}{2},T\right]}}
+ \|\partial_\nu w_{\mathrm{cont.}}\|_{H^{m}\!\left(\left(\frac{T_{0}}{2},T\right)\times\partial\Omega\right)} \\
&\quad \le
C\,e^{C\left(T-\frac{T_{0}}{2}\right)}
\Biggl(
\sum_{k=0}^{m}
\bigl\|\partial_t^{\,k}\!\left(F_{1}-a\,\partial_t v_{\mathrm{cont.}}\right)
\bigr\|_{L^{1}\left(\frac{T_{0}}{2},T;H^{m-k}(\Omega)\right)} 
+ \left\|w\!\left(\tfrac{T_{0}}{2}\right)\right\|_{H^{m+1}(\Omega)}\\
&\qquad\qquad\qquad\qquad
+ \left\|\partial_t w\!\left(\tfrac{T_{0}}{2}\right)\right\|_{H^{m}(\Omega)}
+ \|f\|_{H^{m+1}\!\left(\left(\frac{T_{0}}{2},T\right)\times\partial\Omega\right)}
\Biggr).
\end{align}
Next, we define the operator
 \begin{align}\label{map:L-epsilon-ext}
\mathcal{\widetilde{T}}_{\left[\frac{T_{0}}{2},T\right]}:\mathcal{F}_{\left[\frac{T_{0}}{2},T\right]}\to\mathcal{F}_{\left[\frac{T_{0}}{2},T\right]},\qquad \mathcal{\widetilde{T}}(v_{\mathrm{cont.}}):=w_{\mathrm{cont.}},
\end{align}
where $w_{\mathrm{cont.}}$ is the unique solution of \eqref{well-posedness-extw} corresponding to  $v_{\mathrm{cont.}}$. Now, we show that this map is a contraction for $v_{\mathrm{cont.}}\in \mathcal{F}_{\left[\frac{T_{0}}{2},\frac{3T_{0}}{2}\right]} $. Let $v_{1,\mathrm{cont.}},v_{2,\mathrm{cont.}}\in\mathcal{F}_{\left[\frac{T_{0}}{2},T\right]}$, and set $w_{i,\mathrm{cont.}}:=\mathcal{\widetilde{T}}_{\left[\frac{T_{0}}{2},T\right]}(v_{i,\mathrm{cont.}})$, $i=1,2$.
Then, following the analysis used in Step \ref{step:1}, we obtain 
\begin{align}
    &\left\|\mathcal{\widetilde{T}}_{\left[\frac{T_{0}}{2},T\right]}(v_{1,\mathrm{cont.}})-\mathcal{\widetilde{T}}_{\left[\frac{T_{0}}{2},T\right]}(v_{2,\mathrm{cont.}})\right\|_{\mathcal{F}_{\left[\frac{T_{0}}{2},T\right]}}= \left\|w_{1,\mathrm{cont.}}- w_{2,\mathrm{cont.}}\right\|_{\mathcal{F}_{\left[\frac{T_{0}}{2},T\right]}}\\
    &\qquad\;\leq C e^{C\left(T-\frac{T_{0}}{2}\right)}
\sum_{k=0}^{m} \left\|\partial_t^k(-a \,\partial_t (v_{1,\mathrm{cont.}}-v_{2,\mathrm{cont.}}))\right\|_{L^1\left(\frac{T_{0}}{2},T;H^{m-k}(\Omega)\right)}\\
&\qquad\leq\;
CC_{m}\left(T-\frac{T_{0}}{2}\right) e^{C\left(T-\frac{T_{0}}{2}\right)}\left\|a\right\|_{H^{m}(\Omega)}
\sum_{k=0}^{m} \left\| \,\partial_t^{k+1} (v_{1,\mathrm{cont.}}-v_{2,\mathrm{cont.}})\right\|_{C\left(\left[\frac{T_{0}}{2},T\right];H^{m-k}(\Omega)\right)}\\
&\qquad \leq CC_{m}\left(T-\frac{T_{0}}{2}\right) e^{C\left(T-\frac{T_{0}}{2}\right)}\left\|a\right\|_{H^{m}(\Omega)}\left\|v_{1,\mathrm{cont.}}-v_{2,\mathrm{cont.}}\right\|_{\mathcal{F}_{\left[\frac{T_{0}}{2},T\right]}}.
\end{align}
Choosing $T= \dfrac{3T_{0}}{2}$, and recalling from Step \ref{step:1} that
$CC_{m} T_{0} e^{C T_{0}}<1$, it follows that
$\widetilde{\mathcal{T}}_{\left[\frac{T_{0}}{2},\frac{3T_{0}}{2}\right]}$
is a contraction. Then by Banach fixed-point Theorem, there exists a unique solution $w_{\mathrm{cont.}}\in \mathcal K_{\left[\frac{T_{0}}{2},\frac{3T_{0}}{2}\right]}$ satisfying
\begin{align}
&\|w_{\mathrm{cont.}}\|_{\mathcal{F}_{\left[\frac{T_{0}}{2},\frac{3T_{0}}{2}\right]}}
+ \|\partial_\nu w_{\mathrm{cont.}}\|_{H^{m}\!\left(\left(\frac{T_{0}}{2},\frac{3T_{0}}{2}\right)\times\partial\Omega\right)} \\
&\quad \le
C\,e^{C T_{0}}
\Biggl(
\sum_{k=0}^{m}
\|\partial_t^{\,k} F_{1}\|_{L^{1}\!\left(\frac{T_{0}}{2},\frac{3T_{0}}{2};H^{m-k}(\Omega)\right)} 
+ \left\|w\!\left(\tfrac{T_{0}}{2}\right)\right\|_{H^{m+1}(\Omega)}\\
&\qquad\qquad\qquad
+ \left\|\partial_t w\!\left(\tfrac{T_{0}}{2}\right)\right\|_{H^{m}(\Omega)} 
+ \|f\|_{H^{m+1}\!\left(\left(\frac{T_{0}}{2},\frac{3T_{0}}{2}\right)\times\partial\Omega\right)}
\Biggr).
\end{align}
Next, we show that $w_{\mathrm{cont.}}$ agrees with $w$ on the common domain. More precisely, we prove that
\begin{equation*}
w \equiv w_{\mathrm{cont.}} \quad \text{in } \left(\tfrac{T_{0}}{2}, T_{0}\right)\times \Omega .
\end{equation*}
This immediately follows  from the fact that both $w$ and $w_{\mathrm{cont.}}$ satisfy the same IBVP that is given by 
\begin{align}\label{overlapping-eq:w}
\begin{cases}
\Box w^{\ast}(t,x) + b(x)\,w^{\ast}(t,x) + a(x)\,\partial_tw^{\ast}(t,x) = F_1(t,x),  &(t,x)\in \left(\frac{T_{0}}{2},T_{0} \right)\times \Omega,\\
w^{\ast}(t,x) = f(t,x), &(t,x)\in \left(\frac{T_{0}}{2},T_{0} \right)\times \partial\Omega,\\
w^{\ast}\left(\frac{T_{0}}{2}, x \right)=w\left(\frac{T_{0}}{2}, x \right),\quad \partial_{t}w^{\ast}\left(\frac{T_{0}}{2}, x \right)=\partial_{t}w\left(\frac{T_{0}}{2}, x \right), &\quad x \in \Omega.
\end{cases}
\end{align}
Hence, by the uniqueness of solutions to IBVP \eqref{eq:w} and \eqref{well-posedness-extw}, we conclude that $w_{\mathrm{cont.}}$ agrees with $w$ on the common domain.
We now define a function $w_{\mathrm{ext}}$ on
$\left(0,\tfrac{3T_{0}}{2}\right)\times\Omega$ by
\begin{align}
w_{\mathrm{ext}}(t,x)=
\begin{cases}
w(t,x), & (t,x)\in (0,T_{0})\times\Omega,\\
w_{\mathrm{cont.}}(t,x),
& (t,x)\in \left(\tfrac{T_{0}}{2},\tfrac{3T_{0}}{2}\right)\times\Omega .
\end{cases}
\end{align}
By construction, $w_{\mathrm{ext}}$ solves the IBVP
\eqref{eq:wave-ab} in $\left(0,\tfrac{3T_{0}}{2}\right)\times\Omega$.
Repetition of this argument inductively, the solution can be extended to the entire
domain $(0,T)\times\Omega$.

\end{step}
\end{proof}

Building on this,  we are now ready to state the main result of this section, which asserts that, for sufficiently small Dirichlet boundary data $f$, the IBVP \eqref{equation; IBVP} admits a unique solution $u$ which belongs to \begin{align*}
    \mathbb{B}_R(0) := \left\{ u \in \mathcal{E}_{m+1} \;:\; \|u\|_{\mathcal{E}_{m+1}} <R\right\},
\end{align*}
for some $R>0$. More precisely, we prove the following. 
\begin{theorem}[Well-posedness for semi-linear Equations with damping, linear and nonlinear potential]\label{main_thm:welposedness}
{\it Let $m>n+1$, $ M >0$ and $T>0$. Assume that $a, b,q\in C_c^{\infty}(\Omega)$  and  $f \in \mathcal{D}^\delta_{m+1}$, with $\partial_t^k f(0)$ belong to $H_{0}^{m-k}(\Omega)$ for $k=0,1,\dots,m-1$, then the IBVP \eqref{equation; IBVP} admits a unique solution $u\in \mathbb{B}_R(0) $ such that 
\begin{align}
  \|u\|_{\mathcal{E}_{m+1}} + \|\partial_\nu u\|_{H^{m}(\Sigma)}
\le C\,e^{CT} \|f\|_{H^{m+1}(\Sigma)}, 
\end{align}
holds, for some constant $C>0$, depending only on $\Omega, a,b$ and $q$. }
\end{theorem}

\begin{proof}

For $R>0$ (to  be specified later), define the closed subspace $\mathcal{D}^{[0,R]}$ of $\mathcal{E}_{m+1}$ by 
\begin{align*}
\mathcal{D}^{[0,R]}
:=
\Big\{
u \in \mathbb{B}_{R}(0):\ \ 
\partial_t^k u(0) \in H_0^{m-k}(\Omega),\ k=0,1,\dots,m-1
\Big\},
\end{align*}  equipped with the norm of $\mathcal{E}_{m+1}$ as defined in  \eqref{norm-Em}. 
For  $u \in \mathcal{D}^{[0,R]}$, consider the following  IBVP
\begin{align}\label{eq; wave-a_b_F_1_rewrite}
\begin{cases}
\Box v(t,x) + a(x)\partial_t v(t,x) + b(x)v(t,x) = -q(x)u^{\ell}(t,x), & (t,x)\in\Omega_T,\\
v(t,x) = f(t,x), & (t,x)\in\Sigma,\\
v(0,x)=0,\quad \partial_t v(0,x)=0, & x\in\Omega.
\end{cases}
\end{align}
 Now since $u \in \mathcal{D}^{[0,R]}$, therefore using the fact that  $\mathcal{E}_{m}$ is a Banach algebra, we get that  $-qu^{\ell} \in \mathcal{E}_{m} $, and using the Leibniz rule for product rule for Sobolev functions, we get that the compatibility conditions are also satisfied. Thus, in view of this, we can apply Lemma~\ref{lem:potential and damping}, to obtain that the IBVP 
\eqref{eq; wave-a_b_F_1_rewrite} admits a unique solution  
$v \in \mathcal{E}_{m+1}$ such that $ \partial_\nu v \in H^m(\Sigma)$ and  the following estimate    
\begin{align}\label{eq:est-v=u_rewrite}
\|v\|_{\mathcal{E}_{m+1}}+\|\partial_\nu v\|_{H^m(\Sigma)}
\le C e^{CT}
\Bigg(
\sum_{k=0}^m
\|\partial_t^k(q u^{\ell})\|_{L^1(0,T;H^{m-k}(\Omega))}
+\|f\|_{H^{m+1}(\Sigma)}
\Bigg)
\end{align}
holds for some constant $C>0$, depending only $\Omega, a,b$ and $q$. 
Next, in order to prove the well-posedness for \eqref{equation; IBVP}, we define the operator $\mathcal{T}^{sem}:\mathcal{D}^{[0,R]}\longrightarrow \mathcal{E}_{m+1},$ by $\mathcal{T}^{sem}(u):=v$, 
where $v$ is the  solution of the IBVP
\eqref{eq; wave-a_b_F_1_rewrite} corresponding to a given
$u\in\mathcal{D}^{[0,R]}$.
By the estimate \eqref{eq:est-v=u_rewrite}, we obtain
\begin{align}\label{est; L^R}
\|\mathcal{T}^{sem}(u)\|_{\mathcal{E}_{m+1}}
=\|v\|_{\mathcal{E}_{m+1}}
\le
C e^{CT}
\Bigg(
\sum_{k=0}^{m}
\|\partial_t^k(q u^{\ell})\|_{L^1(0,T;H^{m-k}(\Omega))}
+
\|f\|_{H^{m+1}(\Sigma)}
\Bigg).
\end{align}
Now using the fact that $u\in \mathcal{D}^{[0,R]}$ and  $\mathcal{E}_{m}$ is a Banach algebra (see \eqref{eq:Banach_algebra}), we arrive at 
\begin{align}\label{est; qu^2}
\|q u^{\ell}\|_{\mathcal{E}_{m}}
\leq
C_{m}^{\ell}\|q\|_{H^{m}(\Omega)}\,\|u\|_{\mathcal{E}_{m}}^{\ell}
\leq
C_{m}^{\ell}R^{\ell}\,\|q\|_{H^{m}(\Omega)},\quad \mbox{for all $u\in\mathcal{D}^{[0,R]}$.}
\end{align}
Next, applying the estimate \eqref{est; qu^2} and using the fact that $f\in \mathcal{D}^\delta_{m+1}$, we substitute these bounds into \eqref{est; L^R} to obtain 
\begin{align}\label{est; L^R_}
\|\mathcal{T}^{sem}(u)\|_{\mathcal{E}_{m+1}}
\le
C e^{CT}\bigl(\delta + \|q\|_{H^{m}(\Omega)}\, C_{m}^{\ell} R^{\ell}\bigr). 
\end{align}
 Now  to ensures that $\|\mathcal{T}^{sem}(u)\|_{\mathcal{E}_{m+1}}<R$, for all
$u\in\mathcal{D}^{[0,R]}$, we choose $R,\delta>0$, sufficiently small such that
\begin{align}\label{estimate of R}
  R^{\ell-1} < \frac{1}{4CC_{m}^{\ell} e^{CT}\,\|q\|_{H^{m}(\Omega)}},
\quad \mbox{and} \quad 
\delta < \frac{R}{4C e^{CT}}.  
\end{align}
\begin{equation*}
\|\mathcal{T}^{sem}(u)\|_{\mathcal{E}_{m+1}}
\le Ce^{CT}\left(\frac{R}{4C e^{CT}} + \frac{\|q\|_{H^{m}(\Omega)}\,RC_{m}^{\ell}}{4CC_{m}^{\ell}e^{CT}\,\|q\|_{H^{m}(\Omega)}}\right)\le \frac{R}{2}<R.
\end{equation*}
Thus, using the above mentioned choice of $R$,  we get that   $\mathcal{T}^{sem}$ maps
$\mathcal{D}^{[0,R]}$ to itself.
Next, we show that $\mathcal{T}^{sem}:\mathcal{D}^{[0,R]}\rightarrow \mathcal{D}^{[0,R]}$  is a contraction map whenever $R>0$ satisfies the inequality \eqref{estimate of R}. 
To do so, we set $v_i:=\mathcal{T}^{sem}(u_i)$,  for $u_i\in\mathcal{D}^{[0,R]}$, $i=1,2$.   Now  observe that 
 $w:=v_1-v_2$ solves the following IBVP 
\begin{align}\label{eq:v1-v2_rewrite}
\begin{cases}
\Box w(t,x) + a(x)\partial_t w(t,x) + b(x)w(t,x)
= -q(x)(u_1^{\ell}(t,x)-u_2^{\ell}(t,x)), & (t,x)\in\Omega_T,\\
w(t,x)=0, & (t,x)\in\Sigma,\\
w(0,x)=0,\quad \partial_t w(0,x)=0, & x\in\Omega.
\end{cases}
\end{align}
An estimate from Equation \eqref{eq:est-v=u_rewrite}  applied to $w$ along with the following expansion for $u_1^{\ell}-u_2^{\ell}$ \begin{align*}
    u_1^{\ell}-u_2^{\ell} = (u_1-u_2)P_{{\ell}-1}(u_1,u_2), \quad \text{where} \ \ P_{{\ell}-1}(u_1,u_2):= \sum\limits_{k=0}^{{\ell}-1} u_1^{{\ell}-1-k}u_2^{k}, 
\end{align*} yields that  
\begin{align*}
\|w\|_{\mathcal{E}_{m+1}}+\|\partial_\nu w\|_{H^m(\Sigma)}
&\le C e^{CT}
\sum_{k=0}^m
\|\partial_t^k(q(u_1^{\ell}-u_2^{\ell}))\|_{L^1(0,T;H^{m-k}(\Omega))} \\
&\le C C_{m} e^{CT}\|q\|_{H^{m}(\Omega)}\|u_1^{\ell}-u_2^{\ell}\|_{\mathcal{E}_{m+1}} \\
&\le  C C_{m}^{2} e^{CT} \|q\|_{H^{m}(\Omega)}
\|u_1-u_2\|_{\mathcal{E}_{m+1}}\|P_{{\ell}-1}(u_1,u_2)\|_{\mathcal{E}_{m+1}}\\
&\le {\ell} C C_{m}^{\ell}e^{CT} R^{{\ell}-1} \|q\|_{H^{m}(\Omega)}
\|u_1-u_2\|_{\mathcal{E}_{m+1}}, 
\end{align*}
where in the above estimate, we have used the fact that $u_i\in \mathcal{D}^{[0,R]}$ for $i=1,2$. 
Combining the above estimate along with the definition of $\mathcal{T}^{sem}$ and $w$, we arrive at 
\begin{align*}
  \|\mathcal{T}^{sem}(u_1)-\mathcal{T}^{sem}(u_2)\|_{\mathcal{E}_{m+1}}
\le {\ell} CC_{m}^{\ell} e^{CT} R^{{\ell}-1} \|q\|_{H^{m}(\Omega)}
\|u_1-u_2\|_{\mathcal{E}_{m+1}} 
\end{align*}
for some constant $C>0$, depending only on $\Omega, a,b$ and $q$. Next to ensures that the mapping $\mathcal{T}^{sem}:\mathcal{D}^{[0,R]}\rightarrow \mathcal{D}^{[0,R]}$ is a contraction, we choose  $R>0$, such that 
 \begin{align*}
    R^{{\ell}-1}<\min\left\{ \frac{1}{4{\ell}CC_{m}^{\ell}}e^{CT}, \frac{1}{4CC_{m}^{\ell}} e^{CT}\|q\|_{H^{m}(\Omega)}\right\}
\end{align*}
Thus, we have
\begin{align*}
  &\|\mathcal{T}^{sem}(u_1)-\mathcal{T}^{sem}(u_2)\|_{\mathcal{E}_{m+1}}
\le {\ell} CC_{m}^{\ell} e^{CT} R^{{\ell}-1} \|q\|_{H^{m}(\Omega)}
\|u_1-u_2\|_{\mathcal{E}_{m+1}} \\ & \quad  \le {\ell} CC_{m}^{\ell} e^{CT} \frac{1}{4CC_{m}^{\ell}\ell e^{CT}\|q\|_{H^{m}(\Omega)}} \|q\|_{H^{m}(\Omega)}
\|u_1-u_2\|_{\mathcal{E}_{m+1}}\\ & \quad  \le {\ell} CC_{m}^{\ell} e^{CT} \frac{1}{4CC_{m}^{\ell} \ell e^{CT}\|q\|_{H^{m}(\Omega)}} \|q\|_{H^{m}(\Omega)}
\|u_1-u_2\|_{\mathcal{E}_{m+1}}  \le \frac{1}{4}\|u_1-u_2\|_{\mathcal{E}_{m+1}}. 
\end{align*}
 Hence, using the Banach fixed-point Theorem, we get that there exists a unique
$u\in\mathcal{D}^{[0,R]}$ such that $\mathcal{T}^{sem}(u)=u$. Now arguing as earlier, we conclude that the fixed point $u \in \mathcal{E}_{m+1}$ is the unique solution to the IBVP 
\eqref{eq; wave-a_b_F_1_rewrite} such that 
 $\partial_\nu u \in H^m(\Sigma)$ and  satisfies the following estimate
\begin{align}\label{estimate;potential and damping terms_rewrite}
\|u\|_{\mathcal{E}_{m+1}}+\|\partial_\nu u\|_{H^m(\Sigma)}
&\le
C e^{CT}
\left(
\sum_{k=0}^{m}
\big\|\partial_t^k(q u^{\ell})\big\|_{L^1(0,T;H^{m-k}(\Omega))}
+\|f\|_{H^{m+1}(\Sigma)}
\right) \nonumber\\
&\le
C e^{CT}\|q u^{\ell}\|_{\mathcal{E}_{m}}
+ C e^{CT}\|f\|_{H^{m+1}(\Sigma)},
\end{align}
for some constant $C>0$, depending only on $\Omega,a,b$ and $q$. Note that while deriving the above estimate, we used the estimate \eqref{eq:est-v=u_rewrite}. 
Finally,  using the Banach algebra estimate  \eqref{est; qu^2} along with the fact that $u\in\mathcal{D}^{[0,R]}$, in the above estimate,  we get 
\begin{align}\label{est:Banach_q_final}
    \|q u^{\ell}\|_{\mathcal{E}_{m}}
\leq C_{m}^{\ell}
R^{{\ell}-1}\,\|q\|_{H^{m}(\Omega)}\,\|u\|_{\mathcal{E}_{m}}. 
\end{align}
 Now, substituting the upper bound of $R$ and $\delta$ from \eqref{estimate of R} along with the estimate \eqref{est:Banach_q_final}, into
 \eqref{estimate;potential and damping terms_rewrite} to obtain
\begin{align*}
& \|u\|_{\mathcal{E}_{m+1}}+\|\partial_\nu u\|_{H^m(\Sigma)}
\leq
C e^{CT}
\left(
\sum_{k=0}^{m}
\big\|\partial_t^k(q u^{\ell})\big\|_{L^1(0,T;H^{m-k}(\Omega))}
+\|f\|_{H^{m+1}(\Sigma)}
\right)\\
& \quad \leq
C e^{CT}\|q u^{\ell}\|_{\mathcal{E}_{m}}
+ C e^{CT}\|f\|_{H^{m+1}(\Sigma)}\le C e^{CT} C_{m}^{\ell} R^{{\ell}-1}\,\|q\|_{H^{m}(\Omega)}\,\|u\|_{\mathcal{E}_{m}}
+  C e^{CT}\|f\|_{H^{m+1}(\Sigma)}\\
& \quad \le C e^{CT}C_{m}^{\ell} \frac{1}{4CC_{m}^{\ell}} e^{CT}\,\|q\|_{H^{m}(\Omega)}\,\|q\|_{H^{m}(\Omega)}\,\|u\|_{\mathcal{E}_{m}}
+  C e^{CT}\|f\|_{H^{m+1}(\Sigma)}\\
& \quad \le \frac{1}{4}\|u\|_{\mathcal{E}_{m+1}}
+  C e^{CT}\|f\|_{H^{m+1}(\Sigma)}.
\end{align*}
Further simplification gives,
\begin{align*}
\|u\|_{\mathcal{E}_{m+1}}+\|\partial_\nu u\|_{H^m(\Sigma)}
&\le \frac{4C}{3}  e^{CT}\|f\|_{H^{m+1}(\Sigma)}. 
\end{align*}
Since $C>0$ is a generic constant, we absorb the factor $\frac{4}{3}$ into $C$ and still denote the resulting constant by $C$. This concludes the proof of the Theorem.
\end{proof}
\section{Asymptotic solutions}\label{sec: Asymptotic solutions}
In this section, we construct asymptotic solutions that play a key role in
the reconstruction of the nonlinear potential $q$. Since this reconstruction of $q$ also relies on the prior reconstruction of the damping coefficient $a$ and the  linear potential $b$, therefore we first
present two Lemmas that introduce geometric optics (GO) solutions for the
IBVP \eqref{eq:wave-ab} and the corresponding
backward wave problem. These solutions are essential for recovering the
damping coefficient and the linear potential.  We refer to \cite{Isakov1991AnIH} for detailed proof of Lemmas \ref{lemma: geom opt sol v_{1}} and \ref{lemma: geom opt sol v} stated below.
\begin{lemma}{\cite[Lemma~~2]{Isakov1991AnIH}}\label{lemma: geom opt sol v_{1}}
  {\it  Let $\Omega$ be an open, connected, and bounded subset of $\mathbb{R}^n$, $n\geq 2$ with a smooth boundary $\partial\Omega$ and $T>0$. For any $(a,\varphi)\in C_{c}^{\infty}(\Omega)\times   C_{c}^{\infty}(\R^{n})$ such that $\supp(\varphi) \cap\Omega=\emptyset$ and $\omega \in \mathbb{S}^{n-1}$. For any $\tau>0$, there exists a solution of \eqref{eq;Notations of first order Linearization eqn} having the form 
  \begin{align}\label{eq: v_{(1)}}
  \begin{split}
 v_k(t,x) = \varphi(x+t\omega)A^{+}(t,x)\exp(i\tau(x\cdot\omega + t)) + R_{k}(t,x,\tau),
 \end{split}
\end{align}
where 
\begin{align*}
     A^{+}(t,x)=\exp\left(-\frac{1}{2}\int_0^ta(x+s\omega)\,ds\right),
\end{align*}
and $R_{k}(t,x,\tau)$ vanishes at initial time; that is,
\begin{align}\label{R_{1} initia}
   R_{k}(0,x,\tau) = \partial_{t}R_{k}(0,x,\tau) =  0   \text{ in } \Omega.
\end{align}
Moreover, $R_{k}(t,x,\tau)$ satisfies the following estimate
\begin{align}\label{eq: r_{(1)}}
  \tau\|R_{k}\|_{L^{2}(\Omega_{T})}+\|\partial_{t} R_{k}\|_{L^{2}(\Omega_{T})} \leq C \|\varphi\|_{H^{3}{(\R^n})}.
\end{align}}
\end{lemma}
\begin{lemma}{\cite[Lemma~~3]{Isakov1991AnIH}}\label{lemma: geom opt sol v}
  {\it  Let $\Omega$ be an open, connected, and bounded subset of $\mathbb{R}^n$, $n\geq 2$ with a smooth boundary $\partial\Omega$ and $T>0$. For any $(a,\varphi)\in C_{c}^{\infty}(\Omega)\times C_{c}^{\infty}(\R^{n})$ such that $\left(\supp(\varphi)\pm T\omega\right)\cap\Omega=\emptyset$  and $\omega \in \mathbb{S}^{n-1}$. For any $\tau>0$, there exists a solution of  \eqref{FBVP; v} having the form 
  \begin{align}\label{eq: v}
  v(t,x) = \varphi(x+t\omega)A^{-}(t,x)\exp(-i\tau(x\cdot\omega + t)) + R_0(t,x,\tau),
\end{align}
where
\begin{align*}
    A^{-}(t,x)=\exp\left(\frac{1}{2}\int_0^ta(x+s\omega)\,ds\right),
\end{align*}
 and $R_0(t,x,\tau)$ vanishes at the final time; that is,
\begin{align}
   R_0(T,x,\tau) = \partial_{t}R_0(T,x,\tau) =  0   \text{ in } \Omega.
\end{align}
Moreover, $R_0(t,x,\tau)$ satisfies the following estimate
\begin{align}\label{eq: estimate of r_{(2)}}
  \tau\|R\|_{L^{2}(\Omega_{T})}+\|\partial_{t} R\|_{L^{2}(\Omega_{T})} \leq C \|\varphi\|_{H^{3}{(\R^n})}.
\end{align}}
\end{lemma}
Now, as explained in subsection \ref{subsection; linear coefficients} below, the recovery of linear coefficients $a$ and $b$ requires that the product of two $L^2$ Sobolev functions on $\Omega_T$, which solve linear PDEs, is dense in $L^1(\Omega_T)$. Therefore, for reconstructing the linear coefficients, it suffices to use the solutions constructed in Lemmas \ref{lemma: geom opt sol v_{1}} and \ref{lemma: geom opt sol v}.
In contrast, the recovery of the nonlinear potential $q$ involves products of more than two $L^2$ Sobolev solutions of the associated linear equations on $\Omega_T$. Such products need not belong to $L^1(\Omega_T)$, which necessitates working with solutions in spaces with stronger regularity. To address this issue, we construct asymptotic solutions whose remainder term admit $L^\infty$ decay estimate, which are suitable for recovering the nonlinear potential $q$. These $L^\infty$ estimates are obtained by considering higher-order expansions of the solution in negative powers of a large parameter $\tau$. For more details, we refer to Propositions \ref{Asymptotic solutions} and \ref{Asymptotic solutions for IBVP} below.
\begin{proposition}[Asymptotic solutions for backward problem]  \label{Asymptotic solutions}
{\it Let $\Omega, \Omega_T, \partial\Omega$, $T>0$ and $(a,b) \in C_c^{\infty}(\Omega)\times C_c^{\infty}(\Omega)$ be as before.  Then the backward wave problem
   \begin{align}\label{pr;backwad_problem}
     \begin{cases}
         \mathcal{L}_{a,b} v(t,x):= \left(\Box - a(x)\partial_t + b(x)\right)v(t,x) = 0, \ \ (t,x) \in \Omega_T,\\
     v(T,x)=\partial_t v (T,x) = 0,  \ \ x \in\Omega,
     \end{cases}
 \end{align}
 admits an asymptotic solution  having the form 
    \begin{equation}\label{eqn: IBVP geom opt sol 1}
 v(t,x)= e^{-{\ell}i\tau\,(t+x \cdot\omega)}\left(\beta_{0}(t,x)+\frac{\beta_{1}(t,x)}{\tau}+\frac{\beta_{2}(t,x)}{\tau^2}+\cdots+\frac{\beta_{N}(t,x)}{\tau^N}\right)+R_0(t,x,\tau),
    \end{equation}
    where the functions $\beta_j$ for $j = 0,1,\dots, N$ satisfy the following transport systems:
The function $\beta_0$ solves
\begin{align}
\begin{cases}
     L\beta_0(t,x):= \left(\partial_t - \omega\cdot\nabla_x - \frac{a(x)}{2}\right)\beta_{0}(t,x) = 0, & (t,x)\in  \Omega_T, \\
    \beta_0(T,x)=0,\quad \partial_t \beta_0(T,x)=0, & \quad x\in  \Omega .
\end{cases}
\end{align}
For $j = 1,\dots, N$, the functions $\beta_j$ solves
\begin{align}\label{eq:transport-mj}
\begin{cases}
   2{\ell}i L \beta_j(t,x) =  \mathcal{L}_{a,b} \beta_{j-1}(t,x), & (t,x)\in  \Omega_T, \\
    \beta_j(T,x)=0,\quad \partial_t \beta_j(T,x)=0, & \quad x\in  \Omega .
\end{cases}
\end{align}
 The correction term $R_0(t,x,\tau)$ satisfies 
\begin{align}\label{eq: r_{(2)}}
  R_0(T,x,\tau)=\partial_t R_0(T,x,\tau)=0,\  x\in  \Omega \  \mbox{and}\ \left\|R_0\right\|_{L^{\infty}(\Omega_{T})}\leq \frac{C}{\tau},\ \mbox{for any $\tau>1$}
\end{align}
where the constant $C>0$ is independent of $\tau$.}
    \end{proposition}

 Before proving Proposition \ref{Asymptotic solutions}, we establish the existence
of the functions $\beta_j$, $j=0,1,\dots,N$, solving the associated transport systems.


\begin{lemma}\label{prop:transport}
{\it For any integer $j = 0,\dots, N$, there exist functions 
    $\beta_{j}\in C^{\infty}\left(  \overline{\Omega}_{T}\right)$
 such that $\beta_0$ solves 
\begin{align}\label{eqn:transport-m0}
\begin{cases}
L \beta_0(t,x) = 0, \ \  &(t,x)\in \Omega_T, \\
\beta_0(T,x)=\partial_t \beta_0(T,x)=0, \ \  &x\in \Omega,
\end{cases}
\end{align}
and for $1\leq j\leq N$, the functions $\beta_j$,  solve
\begin{align}\label{eqn:transport-mj}
\begin{cases}
2{\ell}i\, L \beta_j(t,x) =  \mathcal{L}_{a,b} \beta_{j-1}(t,x), \ \  (t,x) \in \Omega_T, \\
\beta_j(T,x)= \partial_t \beta_j(T,x)=0, \ \  x\in \Omega.
\end{cases}
\end{align}}
\end{lemma}

\begin{proof}
Our strategy to prove this Lemma, is via constructing  explicit solutions to the transport Equations
\eqref{eqn:transport-m0} and \eqref{eqn:transport-mj}. We begin by reformulating the backward problem as an initial value problem (IVP)
through a suitable change of variables. More precisely, we define
\begin{equation}\label{eqn:tildem0tomo}
\widetilde{\beta}_0(t,x) := \beta_0(T-t,x).
\end{equation}
By the chain rule, $\widetilde{\beta}_0$ satisfies the following IVP 
\begin{equation}\label{eqn:tildem0}
\begin{cases}
\displaystyle
\left(-\partial_t - \omega\cdot\nabla_x - \frac{a(x)}{2}\right)\widetilde{\beta}_0(t,x) = 0,\ \ (t,x) \in \Omega_T, \\[8pt]
\widetilde{\beta}_0(0, x )=\partial_t \widetilde{\beta}_0(0, x )=0,\ \  x\in \Omega.
\end{cases}
\end{equation}
Next, we plan to convert the above IVP to one that does not contain the lower-order term. To do so, we introduce the following 
function 
\begin{equation}\label{eqn:m0tov0}
\widetilde{v}_0(t,x) := -\widetilde{\beta}_0(t,x) 
  \,
\exp\left(-\int_{0}^{T-t}\frac{a(x+\tau\omega)}{2}\,d\tau\right),\ (t,x)\in \Omega_T,
\end{equation}
which  implies that 
\begin{align}\label{eqn:m0tov0_revised}
    \widetilde{\beta}_0(t,x) = -\widetilde{v}_0(t,x)\,
\exp\left(\int_{0}^{T-t}\frac{a(x+\tau\omega)}{2}\,d\tau\right), \ (t,x)\in \Omega_T.
\end{align}
Next, we derive an Equation satisfied by $\widetilde{v}_0$ by computing the action of the  operator 
\begin{align*}
    -\partial_t - \omega\cdot\nabla_x - \frac{a(x)}{2}
\end{align*}
to $\widetilde \beta_0(t,x)$. We start with the following computation,
\begin{align} \label{eq:dtm0}
-\partial_t \widetilde \beta_0
= \exp\left(\int_{0}^{T-t}\frac{a(x+\tau\omega)}{2}\,d\tau\right)
\left(
\partial_t \widetilde{v}_0
- \frac{a\left(x+(T-t)\omega\right)}{2} \widetilde{v}_0
\right),
\end{align}
and
\begin{align*}
\begin{aligned}
&-\omega\cdot\nabla_x \widetilde \beta_0
= \exp\left(\int_{0}^{T-t}\frac{a(x+\tau\omega)}{2}\,d\tau\right)
\left(
\omega\cdot\nabla_x \widetilde{v}_0
+ \frac{\widetilde{v}_0}{2}
\int_{0}^{T-t}\omega\cdot\nabla a(x+\tau\omega)\,d\tau
\right). 
\end{aligned}
\end{align*}
The above expression, after using the chain rule and the fundamental Theorem of calculus, is given by 
\begin{equation*}
\int_{0}^{T-t}\omega\cdot\nabla a(x+\tau\omega)\,d\tau=\int_{0}^{T-t}\partial_{\tau}a(x+\tau\omega)\,d\tau
= a\left(x+(T-t)\omega\right)-a(x)
\end{equation*}
implies that 
\begin{align}\label{eq:gradm0}
-\omega\cdot\nabla_x \widetilde \beta_0=\exp\left(\int_{0}^{T-t}\frac{a(x+\tau\omega)}{2}\,d\tau\right)
\left[
\omega\cdot\nabla_x \widetilde{v}_0
+ \frac{\widetilde{v}_0}{2}\Big\{a\left(x+(T-t)\omega\right)-a(x)\Big\}
\right]. 
\end{align} 
Combining \eqref{eq:dtm0} and \eqref{eq:gradm0}, we obtain
\begin{align}
\left(-\partial_t - \omega\cdot\nabla_x -\frac{a(x)}{2}\right)\widetilde \beta_0(t,x)
=
\exp\left(\int_{0}^{T-t}\frac{a(x+\tau\omega)}{2}\,d\tau\right)
\left(
\partial_t \widetilde{v}_0(t,x) + \omega\cdot\nabla_x \widetilde{v}_0(t,x)
\right).
\end{align}
Using  \eqref{eqn:tildem0} along with  non-vanishing of exponential factor implies that  $\widetilde{v}_0$ solves  the IVP
\begin{equation}\label{eq:transport-v0}
\begin{cases}
\partial_t \widetilde{v}_0(t,x) + \omega\cdot\nabla_x \widetilde{v}_0(t,x) = 0,
\  (t,x)\in \Omega_T, \\
\widetilde{v}_0(0, x )=\partial_t \widetilde{v}_0(0, x )=0,\ \  x\in \Omega .
\end{cases}
\end{equation}
We choose 
\begin{equation*}
\widetilde{v}_0(t,x) := -\varphi\left(x + \left(T- t\right)\omega\right),
\end{equation*}
where $\omega\in\mathbb{S}^{n-1}$ and $\varphi \in C^\infty_c(\mathbb{R}^n)$ is such that  
\begin{align}\label{eqn:condt.testfunction}
    \mbox{$\left(\operatorname{supp}(\varphi)\pm T\omega\right)\cap \Omega=\emptyset$.}
\end{align}
With this choice, we observe that $\widetilde{v}_0$ solves IVP \eqref{eq:transport-v0}.
From Equation \eqref{eqn:m0tov0_revised}, we obtain
\begin{equation*}
\widetilde \beta_0(t,x)
= \varphi\left(x + \left(T- t\right)\omega\right)
\exp\left(\int_{0}^{T-t}\frac{a(x+\tau\omega)}{2}\,d\tau\right).
\end{equation*}
Changing to the original variables from \eqref{eqn:tildem0tomo} yields the following
\begin{equation}\label{eqn:m0}
\beta_0(t,x)
= \widetilde \beta_0(T-t,x)
= \varphi(x+ t\omega)
\exp\left(\int_{0}^{t}\frac{a(x+\tau\omega)}{2}\,d\tau\right).
\end{equation}
Next, we construct an explicit solution for $\beta_1$, which satisfies the transport Equation \eqref{eqn:transport-mj} for $j=1$, that is,
\begin{align}\label{eqn:m1-g}
\begin{cases}
L \beta_1(t,x) = g(t,x), \ &(t,x) \in \Omega_T, \\
\beta_1(T,x)=\partial_t \beta_1(T,x)=0, \ \  &x \in  \Omega .
\end{cases}
\end{align}
where \begin{equation}\label{eqn:g-rep.}
g(t,x) := \frac{1}{2{\ell}i} \mathcal{L}_{a,b} \beta_0(t,x).
\end{equation}
We proceed as is done for the construction of $\beta_0$, however, due to the presence of the inhomogeneous term $g$, calculations are slightly more involved. We start by converting the backward problem \eqref{eqn:tildem1} to IVP by introducing  the function $\widetilde \beta_1(t,x)$ by 
\begin{equation*}
 \widetilde \beta_1(t,x) := \beta_1(T-t,x),\ (t,x)\in \Omega_T.
\end{equation*}
Then $\widetilde \beta_1$ satisfies the following IVP
\begin{equation}\label{eqn:tildem1}
\begin{cases}
\left(-\partial_t - \omega\cdot\nabla_x - \dfrac{a(x)}{2}\right)\widetilde \beta_1(t,x)
= g(T-t,x), & (t,x)\in  \Omega_T, \\[10pt]
\widetilde \beta_1(0, x )=0,\quad \partial_t \widetilde \beta_1(0, x )=0,
& \quad x\in  \Omega .
\end{cases}
\end{equation}
Next, we substitute
\begin{equation}\label{eqn:tildem1 rep}
\widetilde \beta_1(t,x)
= - v(t,x)\exp\left(\int_{0}^{T-t}\frac{a(x+\tau\omega)}{2}\,d\tau\right).
\end{equation}
Now, we derive an Equation satisfied by $v$ by computing the action of the operator
$\left(-\partial_t - \omega\cdot\nabla_x - \dfrac{a(x)}{2}\right)$
on $\widetilde \beta_1$, to  get
\begin{align}
\left(-\partial_t - \omega\cdot\nabla_x - \frac{a(x)}{2}\right)\widetilde \beta_1(t,x)
&= \exp\left(\int_{0}^{T-t}\frac{a(x+\tau\omega)}{2}\,d\tau\right)
\left( \partial_t v + \omega\cdot\nabla_x v \right).
\end{align}
After using the IVP \eqref{eqn:tildem1} along with the non-vanishing of the exponential factor, we arrive at 
\begin{equation}
\begin{cases}
\partial_t v(t,x) + \omega\cdot\nabla_x v(t,x)
= g(T-t,x)\exp\left(-\displaystyle\int_{0}^{T-t}\frac{a(x+\tau\omega)}{2}\,d\tau\right),  \ (t,x)\in \Omega_T,
 \\
v(0, x )=\partial_t v(0, x )=0, \ \  x\in\Omega.
\end{cases}
\end{equation}
To simplify the above expression, we define
\begin{equation}\label{def:L}
L(t,x)
:= g(T-t,x)\exp \left(-\int_{0}^{T-t}\frac{a(x+\tau\omega)}{2}\,d\tau\right).
\end{equation}
In addition, define the function
\begin{equation*}
z(s) := v(t+s,\,x+s\omega).
\end{equation*}
Then, using the chain rule
\begin{equation}
\frac{d}{ds} z(s)
= \partial_t v(t+s,x+s\omega)
+ \omega\cdot\nabla v(t+s,x+s\omega)
= L(t+s,x+s\omega).
\end{equation}
Integrating with respect to $s$ from $-t$ to $0$, we obtain
\begin{align}
\int_{-t}^{0} \frac{d}{ds} z(s)\,ds
&= \int_{-t}^{0} L(t+s,x+s\omega)\,ds, \\
z(0)-z(-t)
&= \int_{-t}^{0} L(t+s,x+s\omega)\,ds.
\end{align}
Using the initial condition $v(0, x )=0$, we have $z(-t)=v(0,x-t\omega)=0$ and hence
\begin{equation}\label{eqn:v-repr}
v(t,x)
= \int_{-t}^{0} L(t+s,x+s\omega)\,ds.
\end{equation}
From the definition of $L$ from \eqref{def:L}, we rewrite \eqref{eqn:v-repr} as
\begin{align}\label{value of v(t,x)}
v(t,x)
&= \int_{-t}^{0}
g(T-t-s,x+s\omega)
\exp\left(-\int_{0}^{T-(t+s)}\frac{a(x+(s+\tau)\omega)}{2}\,d\tau\right)
\,ds.
\end{align}
Using  Equations \eqref{eqn:tildem1 rep} and \eqref{value of v(t,x)},  together with the definition of $\widetilde{\beta}_1$, we obtain
\begin{align}
\begin{aligned}
&- \beta_1(T-t,x)\,
\exp\left(-\int_{0}^{T-t}\frac{a(x+\tau\omega)}{2}\,d\tau\right)\\
&\qquad =
\int_{-t}^{0}
g(T-t-s,x+s\omega)
\exp\left(-\int_{0}^{T-(t+s)}\frac{a(x+(s+\tau)\omega)}{2}\,d\tau\right)
\,ds.
\end{aligned} 
\end{align}
Substituting  $g$ from \eqref{eqn:g-rep.}, we get
\begin{align}
&- \beta_1(T-t,x)\,
\exp\left(-\int_{0}^{T-t}\frac{a(x+\tau\omega)}{2}\,d\tau\right)\\
&\qquad=
\int_{-t}^{0}
\frac{1}{2i\ell} \mathcal{L}_{a,b} \beta_0(T-t-s,x+s\omega)
\exp\left(-\int_{0}^{T-(t+s)}\frac{a(x+(s+\tau)\omega)}{2}\,d\tau\right)
\,ds.
\end{align}
Finally, using the transformation $t\mapsto T-t$ yields the explicit representation
\begin{align}\label{eqn:m1-explicit}
\beta_1(t,x)
&=
-\int_{-(T-t)}^{0}\frac{1}{2i\ell}
 \left(\mathcal{L}_{a,b} \beta_0(t-s,x+s\omega) 
\exp\left(-\int_{0}^{t-s}\frac{a(x+(s+\tau)\omega)}{2}\,d\tau\right)\right)
\,ds\\
& \qquad \times \exp\left(\int_{0}^{t}\frac{a(x+\tau\omega)}{2}\,d\tau\right).
\end{align}
By construction, the function $\beta_1$ satisfies the final-time condition; that is,
\begin{equation*}
\beta_1(T,x)=0.
\end{equation*}
Next, we show that $\beta_1$ also satisfies
\begin{equation*}
\partial_t \beta_1(T,x)=0.
\end{equation*}
For notational convenience, we introduce
\begin{align*}
E(t,x)
&:= \exp\left(\int_{0}^{t}\frac{a(x+\tau\omega)}{2}\,d\tau\right), \\
{E}_1(t,s,x)
&:= \exp\left(-\int_{0}^{t-s}\frac{a(x+(s+\tau)\omega)}{2}\,d\tau\right), \\
F(t,s,x)
&:= \frac{1}{2i\ell}\mathcal{L}_{a,b} \beta_0(t-s,x+s\omega).
\end{align*}
Now, using the above notations, we observe that the  function $\beta_1$  can be written as
\begin{equation}\label{eq:m1compact}
\beta_1(t,x)
=
-E(t,x)\int_{-(T-t)}^{0} F(t,s,x)\,{E}_1(t,s,x)\,ds.
\end{equation}
Differentiating  with respect to $t$ and applying the product rule, we obtain
\begin{align}\label{eq:dtm1}
\partial_t \beta_1(t,x)
&=
-E(t,x)\,
\partial_t\left(
\int_{-(T-t)}^{0} F(t,s,x)\,{E}_1(t,s,x)\,ds\right) \\
&\quad-
(\partial_t E)(t,x)
\int_{-(T-t)}^{0} F(t,s,x)\,{E}_1(t,s,x)\,ds.
\end{align}
Using the Leibniz rule, we get the following expression
\begin{align}
\partial_t \beta_1(t,x)
&=
E(t,x)\,F(t,-(T-t),x)\,{E}_{1}(t,-(T-t),x)-E(t,x)\,\int_{-(T-t)}^{0}
\partial_t\!\left(F\ {E}_1\right)(t,s,x)\,ds
 \nonumber\\
&\qquad \qquad-
(\partial_t E)(t,x)
\int_{-(T-t)}^{0} F(t,s,x)\,{E}_1(t,s,x)\,ds.
\end{align}
Evaluating the above expression at $t=T$, we obtain
\begin{align}\label{eqn:intermediatedtm1}
\partial_t \beta_1(T,x)
&=
E(T,x)\,F(T,0,x)\,{E}_{1}(T,0,x).
\end{align}
To establish that the above expression vanishes, we will show that $F(T,0,x)=0,$ for $x\in\Omega$. To verify this, we use condition \eqref{eqn:condt.testfunction} together with the explicit formula for $\beta_0$ given in \eqref{eqn:m0}, we conclude that 
\begin{equation*}
\partial_t^{k} \beta_0(T,x)= 0,
\ \  \text{ for every } k \in \mathbb{N}\cup\{0 \} \  \mbox{and}\      x\in\Omega.
\end{equation*}
Therefore, we have
\begin{equation*}
\mathcal{L}_{a,b} \beta_0(T,x)=0,
\end{equation*}
and since $\mathcal{L}_{a,b} \beta_0(T,x)=F(T,0,x)$, it follows from \eqref{eqn:intermediatedtm1} that
\begin{equation*}
\partial_t \beta_1(T,x)=0.
\end{equation*}
Now by the repeated use of product rule in \eqref{eq:dtm1} together with the fact that $\partial_{t}^{k}F(T,0,x)=0$, for every $k\in \mathbb{N}\cup \{0\}$, $x\in \Omega$  gives
\begin{equation}\label{eq:beta1_all_time_derivatives_vanish_at_T}
\partial_t^{k}\beta_{1}(T,x)=0,
\ \ \text{for every } k\in\mathbb{N}\cup\{0\}, \ \ x\in\Omega.
\end{equation}
Using calculations similar to those for $\beta_{1}$, we can define $\beta_{j}$, $j = 2,\dots, N$, as follows
\begin{align}\label{eqn:mj-explicit}
\beta_j(t,x)
&=
-\int_{-(T-t)}^{0}\frac{1}{2i\ell}
 \Bigg(\mathcal{L}_{a,b} \beta_{j-1}(t-s,x+s\omega) 
\exp\left(-\int_{0}^{t-s}\frac{a(x+(s+\tau)\omega)}{2}\,d\tau\right)\Bigg)
\,ds \\
& \qquad \times 
\exp\left(\int_{0}^{t}\frac{a(x+\tau\omega)}{2}\,d\tau\right).
\end{align}
By construction, $\beta_j$, $j = 2,\dots, N$, satisfies the final-time condition 
\begin{equation*}
\beta_j(T,x)=0.
\end{equation*}
Moreover, the vanishing of the time derivative at the final time
\begin{equation*}
\partial_t \beta_j(T,x)=0, \quad j = 2,\dots, N,
\end{equation*}
followed by an argument similar to the one used to establish 
$\partial_t \beta_{1}(T,x)=0$ together with \eqref{eq:beta1_all_time_derivatives_vanish_at_T}. This concludes the proof of the Lemma. 
\end{proof}
Next, using the construction of $\beta_j$ for $0\leq j\leq N$, established in the above Lemma, we are in a position to prove  Proposition \ref{Asymptotic solutions}.
\subsection{Proof of Proposition \ref{Asymptotic solutions}}
We start  by recalling the product rule for the $\Box$ operator
\begin{align}
    \Box (f g) = f\Box g + 2\partial_t f\partial_t g - 2 \nabla_x f\cdot\nabla_x g + g\Box f,\quad \text{ for } f,g\in C^{\infty}(\Omega_{T}).
\end{align}
Since $v$ is of the form \eqref{eqn: IBVP geom opt sol 1}, using the above product rule  combined with the fact that $\Box(e^{-{\ell}i\tau(t+x\cdot\omega)})=0$, we have
\begin{align}
\begin{aligned}
\mathcal{L}_{a,b}v(t,x)
&=\mathcal{L}_{a,b}\left(e^{-{\ell}i\tau(t+x\cdot\omega)}
\sum\limits_{j=0}^{N}\tau^{-j} \beta_j(t,x) + R_0(t,x,\tau)\right)\\&=  e^{-{\ell}i\tau(t+x\cdot\omega)}
\sum_{j=0}^{N}\tau^{-j}\mathcal{L}_{a,b}\beta_j(t,x)
- 2{\ell}i\tau e^{-{\ell}i\tau(t+x\cdot\omega)}
\sum_{j=0}^{N}\tau^{-j}
(\partial_t-\omega\cdot\nabla_x)\beta_j(t,x) \\
&\quad
+ {\ell}i\tau a(x)\,e^{-{\ell}i\tau(t+x\cdot\omega)}
\sum_{j=0}^{N}\tau^{-j} \beta_j(t,x)
+ \mathcal{L}_{a,b}R_0(t,x,\tau).
\end{aligned} 
\end{align}
Collecting terms with respect to the powers of $\tau$, we obtain
\begin{align}
\begin{aligned}
\mathcal{L}_{a,b}v(t,x)
&= e^{-{\ell}i\tau(t+x\cdot\omega)}
\Bigg\{
-2{\ell}i\tau L \beta_0(t,x)
+ \left(\mathcal{L}_{a,b}\beta_0 -2{\ell}i L \beta_1\right)(t,x) + \frac{1}{\tau}
\left(\mathcal{L}_{a,b}\beta_1 -2{\ell}i L \beta_2\right)(t,x) \\
& \ \ + \frac{1}{\tau^2}
\left(\mathcal{L}_{a,b}\beta_2 -2{\ell}i L \beta_3\right)(t,x)
+ \cdots
+ \frac{1}{\tau^N}\mathcal{L}_{a,b}\beta_N(t,x)
\Bigg\}  + \mathcal{L}_{a,b}R_0(t,x,\tau).
\end{aligned}
\end{align}
Using the transport Equations \eqref{eq:transport-mj}, all terms of order
$\tau^{-k}$ for $0\leq k\leq N-1$ vanish. Consequently, we obtain
\begin{align}
\mathcal{L}_{a,b}v(t,x)
&= e^{-{\ell}i\tau(t+x\cdot\omega)}
\frac{1}{\tau^{N}}\mathcal{L}_{a,b} \beta_{N}(t,x)
+ \mathcal{L}_{a,b}R_0(t,x,\tau).
\end{align}
Now in order to show that  $v$  solves   $\mathcal{L}_{a,b} v = 0$ in $\Omega_T$ and  $v(T,x)=\partial_t v(T,x)=0$, for $x\in\Omega$,  we conclude  that  $R_0$  must satisfies 
the following backward problem 
\begin{align}
\begin{cases}
\mathcal{L}_{a,b} R_0(t,x,\tau)
= - \tau^{-N}
e^{-{\ell}i\tau(t+x\cdot\omega)}
\mathcal{L}_{a,b} \beta_N(t,x),
\  (t,x)\in  \Omega_T, \\
R_0(T,x,\tau)=\partial_t R_0(T,x,\tau)=0,\ \  x\in  \Omega .
\end{cases}
\end{align}
Using the change of variables $\widetilde t = T - t$, the correction term $R_0$ satisfies the following IVP
\begin{align}
\begin{cases}
\mathcal{L}_{-a,b}\, R_0(\widetilde t,x,\tau)
= - \tau^{-N}
e^{-{\ell}i\tau(-\widetilde t+ T+x\cdot\omega)}
\mathcal{L}_{-a,b} \beta_N(\widetilde t,x),
& (t,x)\in \Omega_T, \\
R_0(0, x,\tau )=\partial_{t} R_0(0, x,\tau )=0,
& x\in \Omega .
\end{cases}
\end{align}
Based on the well-posedness of the IBVP \eqref{lem:potential and damping}, together with $\cap_{k=0}^{\infty}H^{k}(\Omega_{T})= C^{\infty}(\Omega_{T})$ (application of Sobolev embedding), it implies that the correction term
is $R_0 \in C^{\infty}(\Omega_T)$. Moreover, standard energy estimates yield the following bound
\begin{align}
\|R_0\|_{H^{m}(\Omega_T)}\leq C \tau^{2m-N},\qquad \text{ for every }  m\in\mathbb{N},
\end{align}
where the constant $C>0$ is independent of $\tau$ and $\tau>1$.
Choose $m>n+1$, then the Sobolev embedding Theorem implies
\begin{align}
\|R_0\|_{L^{\infty}(\Omega_T)}
\leq \|R_0\|_{H^{m}(\Omega_T)}
\leq C \tau^{2m-N}.
\end{align}
Choosing $N$ sufficiently large,  we conclude that
\begin{align}
\|R_0\|_{L^{\infty}(\Omega_T)}\leq C \tau^{-1},
\end{align} holds for any $\tau>1$. 
This completes the proof. \qed

\vspace{.2cm}
 Next, we present a proposition that establishes the asymptotic solution for the IVP \eqref{pr;Forward_problem}, proof of which follows by using the arguments similar to those used in the proof of Proposition~\ref{Asymptotic solutions}.
\begin{proposition}[Asymptotic solutions for IVP]  \label{Asymptotic solutions for IBVP}
{\it Let $\Omega, \Omega_T, \partial\Omega$, $T>0$ and $(a,b) \in C_c^{\infty}(\Omega)\times C_c^{\infty}(\Omega)$ be as before. Then the following initial   wave problem
   \begin{align}\label{pr;Forward_problem}
     \begin{cases}
         \mathcal{J}_{a,b} v(t,x):= \left(\Box + a(x)\partial_t + b(x)\right)v(t,x) = 0, \ \ (t,x) \in \Omega_T,\\
     v(0,x)=\partial_t v (0,x) = 0,  \ \  x \in\Omega,
     \end{cases}
 \end{align}
 admits an asymptotic solution  having the form 
    \begin{equation}\label{eqn: IVP geom opt sol 1}
 v(t,x)= e^{i\tau\,(t+x \cdot\omega)}\left(m_{0}(t,x)+\frac{m_{1}(t,x)}{\tau}+\frac{m_{2}(t,x)}{\tau^2}+\cdots+\frac{m_{N}(t,x)}{\tau^N}\right)+R(t,x,\tau),
    \end{equation}
    where the functions $m_j$ for $j = 0,1,\dots, N$, satisfy the following transport systems: 
The function $m_0$ solves
\begin{align}
\begin{cases}
     Jm_0(t,x):= \left(\partial_t - \omega\cdot\nabla_x + \frac{a(x)}{2}\right)m_{0}(t,x) = 0, \ \ (t,x)\in  \Omega_T, \\
    m_0(0,x)=\partial_t m_0(0,x)=0, \ \  x\in  \Omega .
\end{cases}
\end{align}
For $j = 1,\dots, N$, the functions $m_j$ solves
\begin{align}\label{eq:transport--mj}
\begin{cases}
   2i J m_j(t,x) =  \mathcal{J}_{a,b} m_{j-1}(t,x),\ \ (t,x)\in  \Omega_T, \\
    m_j(0,x)= \partial_t m_j(0,x)=0, \ \ x\in  \Omega .
\end{cases}
\end{align}
The correction term $R(t,x,\tau)$ satisfies 
\begin{align}\label{eq: r__{(2)}}
    R(0,x,\tau)=\partial_t R(0,x,\tau)=0,\ \text{for}\   x\in  \Omega\ \text{and}\  \|R\|_{L^{\infty}(\Omega_{T})}\leq \frac{C}{\tau}, \ \text{for any $\tau>1$}
\end{align}
where $C>0$ is independent of $\tau$.} 
    \end{proposition}

    \section{Proof of {T}heorem \ref{th: main result}}\label{Sec: main result}
    In this section, we prove the main result stated in  Theorem~\ref{th: main result} of the current article. Our proof is based on  linearization technique, which has been used in prior works as well, see for example \cite{CARSTEA2019121,Isakov1993OnUI,isakov2006inverse,Jiang2023InversePF,Kian2018OnTD,LLLS,NakamuraWatanabe2008,Nakamura2020InverseIB} and references therein.  We start by giving a reconstruction of the linear coefficients using techniques involving first-order linearization. Next, using the reconstruction of linear coefficients along with higher-order linearization techniques, we provide a reconstruction of the nonlinear potential coefficient $q$. The current section is divided into two subsections. The first subsection presents a reconstruction formula for linear coefficients, while the second subsection focuses on deriving a reconstruction formula for nonlinear coefficients.
    Before proceeding further, we mention that throughout this section, we denote  a subspace $\mathfrak{C}_c^{\infty}(\mathbb{R}^n)\subset C_c^{\infty}(\mathbb{R}^n)$ to represent  
\begin{align}\label{eq; definition of C(Rn)}
\mathfrak{C}_c^{\infty}(\mathbb{R}^n):=\{g\in C_c^{\infty}(\mathbb{R}^n):supp(g)\cap\Omega=\emptyset\ \mbox{and}\ (supp(g)\pm T\omega)\cap \Omega=\emptyset,\ \forall \ \omega\in\mathbb{S}^{n-1}\}.
\end{align}
\subsection{Reconstruction of the linear coefficients}\label{subsection; linear coefficients}
As mentioned, in this subsection, we derive reconstruction formulae for the damping coefficient $a$ and the linear potential $b$ using the first-order linearization technique. 
Now for given  $f_1,f_2,\dots,f_{\ell} \in H^{m+1}(\Sigma)$, choose $\epsilon = (\epsilon_1,\dots,\epsilon_{\ell})$ with each $\epsilon_j\geq 0$, for $1\leq j\leq l$,  such that  $\epsilon f$ defined by 
\begin{align*}
  \epsilon f := \epsilon_1 f_1 +\dots + \epsilon_{\ell} f_{\ell},\quad 
\end{align*} 
$\mbox{satisfies  $\epsilon f \in \mathcal{D}^\delta_{m+1}$}.$
 Next  denote $u(t,x):=u_{\epsilon f}(t,x)\in \mathcal{E}_{m+1}$, 
a solution to the  following IBVP
\begin{align}\label{eq;Linearization}
\begin{cases}
\Box u(t,x) + a(x)\partial_t u(t,x) + b(x)u(t,x) + q(x)u^{\ell}(t,x)=0, 
&(t,x)\in  \Omega_T,\\
u(t,x)=\epsilon f(t,x), &(t,x)\in \Sigma,\\
u(0,x)=0,\quad \partial_t u(0,x)=0 ,& \quad x\in  \Omega. 
\end{cases}
\end{align}
Now,  differentiate \eqref{eq;Linearization}  with respect to $\epsilon_k$ for $k=1,\dots,{\ell}$,  to obtain
\begin{align}\label{eq;first order Linearization eqn}
\begin{cases}
\Box (\partial_{\epsilon_k}u(t,x))
+ a(x)\partial_t(\partial_{\epsilon_k}u(t,x))
+ b(x)\,\partial_{\epsilon_k}u(t,x)
+ \ell\,q(x)\,u^{\ell-1}_{\epsilon f}(t,x)\,\partial_{\epsilon_k}u(t,x)
=0,
 (t,x)\in  \Omega_T,\\
\partial_{\epsilon_k}u(t,x) = f_k (t,x),   \quad (t,x)\in  \Sigma,\\
\partial_{\epsilon_k}u(0,x)=\partial_t(\partial_{\epsilon_k}u)(0,x)=0,
\quad x\in  \Omega.
\end{cases}
\end{align}
The interchange of the derivative in the  IBVP \eqref{eq;first order Linearization eqn} can be justified by the definition of weak derivative; for more details, see \cite[Theorem~1 on pp.~247]{evans2022partial}. Next, evaluate the derivative at $\epsilon_k=0$ in the IBVP \eqref{eq;first order Linearization eqn}  and invoking the well-posedness of IBVP \eqref{eq;Linearization},  we conclude that 
$u(t,x)\big|_{\epsilon=0}=0$. Consequently, we obtain
\begin{align}\label{eq;subfirst order Linearization eqn}
\begin{cases}
\Box (\partial_{\epsilon_k}u(t,x))\big|_{\epsilon=0}
+ a(x)\partial_t(\partial_{\epsilon_k}u(t,x))\big|_{\epsilon=0}
+ b(x)(\partial_{\epsilon_k}u(t,x))\big|_{\epsilon=0}=0,
& (t,x)\in  \Omega_T,\\
(\partial_{\epsilon_k}u(t,x))\big|_{\epsilon=0}=f_k(t,x), 
& (t,x)\in  \Sigma,\\
(\partial_{\epsilon_k}u)(0,x)\big|_{\epsilon=0}=0,\quad
\partial_t(\partial_{\epsilon_k}u)(0,x)\big|_{\epsilon=0}=0,
& \quad x\in  \Omega.
\end{cases}
\end{align}
Now, if we denote 
\begin{equation}\label{eq:def-vk}
v_{k}(t,x):=\partial_{\epsilon_k}u(t,x)\big|_{\epsilon=0}, \ \ 1\leq k\leq \ell,
\end{equation}
then IBVP \eqref{eq;subfirst order Linearization eqn} can be rewritten as
\begin{align}\label{eq;Notations of first order Linearization eqn}
\begin{cases}
\Box v_{k}(t,x) + a(x)\partial_t v_{k}(t,x) + b(x)v_{k}(t,x) = 0,
& (t,x)\in  \Omega_T,\\
v_{k}(t,x) = f_k(t,x), & (t,x)\in  \Sigma,\\
v_{k}(0,x)=0,\quad \partial_t v_{k}(0,x)=0, 
&\quad x\in  \Omega,
\end{cases}
\end{align}
for any $1\leq k\leq \ell$. 
Now using the fact that  $u \in \mathcal{E}_{m+1}$ and applying a reasoning similar to the one above, we can interchange the mixed derivatives, to obtain that 
\begin{align}
\partial_{\epsilon_k} \nabla u(t,x,\epsilon_1,\dots,\epsilon_{\ell})= \nabla \partial_{\epsilon_k} u(t,x,\epsilon_1,\dots,\epsilon_{\ell}) \quad \text{for} \ (t,x)\in \Omega_T.
\end{align}
Now we restrict the above identity to $(t,x)\in \Sigma$ and evaluate at $\epsilon = 0$, to arrive at 
\begin{align}
\partial_\nu v_k(t,x)
= \left.\partial_{\epsilon_k}\partial_\nu u(t,x,\epsilon_1,\dots,\epsilon_{\ell})\right|_{\Sigma,\, \epsilon=0}.
\end{align}
Thus, the DN map
associated with the IBVP
\eqref{eq;Notations of first order Linearization eqn} denoted by $\Lambda_{a,b,q}^{(v_k)}$, 
is given by
\begin{align}\label{eq:DN-map-v_k}
\Lambda_{a,b,q}^{(v_k)} (f_k)\big|_{\Sigma}:=
\left[\partial_{\epsilon_k}
\Lambda_{a,b,q}(\epsilon_1 f_1 +\dots+ \epsilon_{\ell} f_{\ell})
\big|_{\Sigma}\right]\big|_{\epsilon=0}
=
\left[\left(\partial_{\epsilon_k}
\partial_\nu u_{\epsilon f}\right)
\big|_{\Sigma}\right]\big|_{\epsilon=0}
=
\partial_\nu v_k \big|_{\Sigma}, 
\end{align}
for all $\ 1\leq k\leq \ell.$ Hence, in view of this, we get that knowing the DN map $\Lambda_{a,b,q}$ corresponding to the IBVP \eqref{equation; IBVP}, gives that  $\partial_\nu v_k\big|_{\Sigma}$ is known for each $1\leq k\leq \ell$. 
 Next, our aim is to establish an integral identity which relates the unknown linear coefficients to the known data. For this, we  multiply IBVP \eqref{eq;Notations of first order Linearization eqn}
 by $v_{0}$, where $v_{0}$ is a  solution of the following backward problem 
 \begin{align}\label{FBVP; v}
\begin{cases}
\Box v_{0}(t,x) = 0, & (t,x)\in \Omega_T, \\
v_{0}(T,x) = 0, \quad \partial_t v_{0}(T,x) = 0, & \quad x \in \Omega,
\end{cases}
\end{align}
and integrate over $\Omega_T$, to obtain
\begin{align}
    \int_{\Omega_T} \left(\Box\, v_k(t,x) + a(x)\partial_t\,v_k(t,x) + b (x)v_k(t,x)\right)v_{0}(t,x)\,dx\,dt = 0.
\end{align}
Now, using integration by parts, we have
\begin{align}\label{eq; a(x) and b(x) in terms of known data}
&\quad\int_{\Omega_T}\big[-a(x)\partial_t\,v_k(t,x)v_{0}(t,x)- b (x)v_k(t,x)v_{0}(t,x)\big]\,dx\,dt \\ &
\quad= \int_{\Omega_T} v_{k}(t,x) \Box\, v_{0}(t,x)\,dx\,dt +\int_{\Omega}\big[\partial_t v_{k}(T,x)v_{0}(T,x)- \partial_t v_{0}(T,x)v_k(T,x)\big]\,dx \\ &
\qquad - \int_{\Omega}\big[\partial_t v_{k}(0,x)v_{0}(0,x)- \partial_t v_{0}(0,x)v_k(0,x)\big]\,dx - \int_{\Sigma}v_{0}(t,x)\partial_{\nu} v_k(t,x)\,dS_x\,dt \\ &
\qquad+\int_{\Sigma} v_k(t,x)\partial_{\nu} v_{0}(t,x) \,dS_x\,dt. 
\end{align}
 On the right-hand side of Equation \eqref{eq; a(x) and b(x) in terms of known data}, the first three integrals vanish using \eqref{eq;Notations of first order Linearization eqn} and \eqref{FBVP; v}, and the last two integrals are known using the DN map \eqref{eq:DN-map-v_k}. Thus, we have the following integral identity
\begin{align}\label{integral identity = known}
\int_{\Omega_T}\big[a(x)\partial_t\,v_k(t,x)v_{0}(t,x)+ b (x)v_k(t,x)v_{0}(t,x)\big]\,dx\,dt = \text{known},
\end{align}
for all $v_k$ ($1\leq k\leq \ell$) and $v_0$, solving \eqref{eq;first order Linearization eqn} and \eqref{FBVP; v} respectively. 

Now the idea is to substitute the special solutions for $v_k$ ($1\leq k\leq \ell$) and $v_0$, and obtain the $X-$ray transforms of the coefficients, which, via the techniques from Fourier analysis, give the reconstruction of unknown coefficients. 

Based on the above idea, we substitute the  solutions constructed in
Lemma~\ref{lemma: geom opt sol v_{1}} for $v_k$  solving \eqref{eq;Notations of first order Linearization eqn} and in Lemma~\ref{lemma: geom opt sol v} for $v_0$ solving \eqref{FBVP; v}, 
into the integral identity~\eqref{integral identity = known} and using the following notation 
\begin{align}\label{eq:R11}
\qquad R_{11}
:=\;&
\left(
\omega\cdot\nabla\varphi(x+t\omega)\,A^{+}(t,x)
+ \varphi(x+t\omega)\,\partial_t A^{+}(t,x)
\right)
e^{i\tau(x\cdot\omega+t)}   + \partial_t R_{k}.
\end{align} along with the fact that $A^{-}\equiv 1$,  in the expression of solution given by \eqref{eq: v},   we obtain that 
\begin{align}\label{eqn:Reconst-a}
\begin{aligned} 
\int_{\Omega_T}
\Big[
& a(x)\left(
i\tau\,\varphi(x+t\omega)A^{+}(t,x)e^{i\tau(x\cdot\omega+t)}
+ R_{11}(t,x,\tau)
\right) \\
&\quad +
b(x)\left(
\varphi(x+t\omega)A^{+}(t,x)e^{i\tau(x\cdot\omega+t)}
+ R_{k}(t,x,\tau)
\right)
\Big] \\
&\qquad \times
\left(
\varphi(x+t\omega)e^{-i\tau(x\cdot\omega+t)}
+ R_0(t,x,\tau)
\right)
\,dx\,dt
= \text{known},\ \mbox{for all $\varphi\in \mathfrak{C}_c^{\infty}(\mathbb{R}^n)$ and $\omega\in\mathbb{S}^{n-1}$,} 
\end{aligned} 
\end{align}  where we refer to Equation \eqref{eq; definition of C(Rn)} for the definition of  $\mathfrak{C}_c^{\infty}(\mathbb{R}^n)$. 
Expanding the integrand in Equation~\eqref{eqn:Reconst-a}, we get 
\begin{align}\label{eq:expanded-identity}
\begin{aligned}
&i\tau \int_{\Omega_T}
a(x)\,\varphi^{2}(x+t\omega)\,A^{+}(t,x)\,dx\,dt  + i\tau \int_{\Omega_T}
a(x)\,\varphi(x+t\omega)\,A^{+}(t,x)
e^{i\tau(x\cdot\omega+t)}R_0(t,x,\tau)\,dx\,dt \\
& \quad+ \int_{\Omega_T}
a(x)\,\varphi(x+t\omega)
e^{-i\tau(x\cdot\omega+t)}R_{11}(t,x,\tau)\,dx\,dt + \int_{\Omega_T}
a(x)\,R_{11}(t,x,\tau)R_0(t,x,\tau)\,dx\,dt \\
& \quad+ \int_{\Omega_T}
b(x)\,\varphi^{2}(x+t\omega)\,A^{+}(t,x)\,dx\,dt + \int_{\Omega_T}
b(x)\,\varphi(x+t\omega)\,A^{+}(t,x)
e^{i\tau(x\cdot\omega+t)}R_0(t,x,\tau)\,dx\,dt \\
& \quad+ \int_{\Omega_T}
b(x)\,\varphi(x+t\omega)
e^{-i\tau(x\cdot\omega+t)}R_{k}(t,x,\tau)\,dx\,dt + \int_{\Omega_T}
b(x)\,R_{k}(t,x,\tau)R_0(t,x,\tau)\,dx\,dt = \text{known}, 
\end{aligned}
\end{align}
for all $\varphi\in \mathfrak{C}_c^{\infty}(\mathbb{R}^n)$ and $\omega\in\mathbb{S}^{n-1}$, where in the above equation, we have used the fact that $\omega\in\mathbb{S}^{n-1}$ can be chosen arbitrary in the GO solutions constructed in Lemmas \ref{lemma: geom opt sol v} and \ref{lemma: geom opt sol v_{1}}. 
Next after dividing \eqref{eq:expanded-identity} by $i\tau$, we obtain
\begin{align}\label{eq:final-splitting}
\int_{\Omega_T}
a(x)\,\varphi^{2}(x+t\omega)\,A^{+}(t,x)\,dx\,dt
+ \sum_{j=2}^{8} I_j =\;& \text{known},\ \mbox{for all $\varphi\in \mathfrak{C}_c^{\infty}(\mathbb{R}^n)$ and $\omega\in\mathbb{S}^{n-1}$}, 
\end{align}
where for $2\leq j\leq 8$, the terms $I_j'$s are given by 
\begin{align*}
   I_2 &=
\int_{\Omega_T}
a(x)\,\varphi(x+t\omega)\,A^{+}(t,x)
e^{i\tau(x\cdot\omega+t)}R_0(t,x,\tau)\,dx\,dt,\\
I_3 &=
\frac{1}{i\tau}\int_{\Omega_T}
a(x)\,\varphi(x+t\omega)
e^{-i\tau(x\cdot\omega+t)}R_{11}(t,x,\tau)\,dx\,dt,\\
I_4 &=
\frac{1}{i\tau}\int_{\Omega_T}
a(x)\,R_{11}(t,x,\tau)R_0(t,x,\tau)\,dx\,dt,\\
I_5 &=
\frac{1}{i\tau}\int_{\Omega_T}
b(x)\,\varphi^{2}(x+t\omega)\,A^{+}(t,x)\,dx\,dt,\\
I_6 &=
\frac{1}{i\tau}\int_{\Omega_T}
b(x)\,\varphi(x+t\omega)\,A^{+}(t,x)
e^{i\tau(x\cdot\omega+t)}R_0(t,x,\tau)\,dx\,dt,\\
I_7 &=
\frac{1}{i\tau}\int_{\Omega_T}
b(x)\,\varphi(x+t\omega)
e^{-i\tau(x\cdot\omega+t)}R_{k}(t,x,\tau)\,dx\,dt,\\
I_8 &=
\frac{1}{i\tau}\int_{\Omega_T}
b(x)\,R_{k}(t,x,\tau)R_0(t,x,\tau)\,dx\,dt.
\end{align*}
Using H\"older's inequality, the estimates
\eqref{eq: r_{(1)}}, \eqref{eq: estimate of r_{(2)}},
and the fact that $a,b\in C_c^\infty(\Omega)$, we estimate each of the integrals
$I_2,\dots,I_8$ appearing in~\eqref{eq:final-splitting}. We write $\|v\| \lesssim \|w\|$ to denote the inequality $\|v\| \leq c\,\|w\|$ for some
constant $c>0$ independent of $v$ and $w$. This notation will be used throughout the rest of the article.

\vspace{.2cm}
\noindent\textbf{Estimate of $I_2$.}
\begin{align*}
|I_2|
&\le
\int_{\Omega_T}
\left|a(x)\varphi(x+t\omega)A^{+}(t,x)R_0(t,x,\tau)\right|
\,dx\,dt \\
&\lesssim
\|\varphi\|_{L^{2}(\R^n)}
\|R\|_{L^2(\Omega_T)}
\lesssim
\frac{1}{\tau}\|\varphi\|^{2}_{H^{3}(\R^n)}.
\end{align*}

\noindent\textbf{Estimate of $I_3$.}
Recalling the expression of $R_{11}$ in~\eqref{eq:R11},
we write
\begin{align*}
|I_3|
&\le
\frac{1}{\tau}
\int_{\Omega_T}
\left|a(x)\varphi(x+t\omega)R_{11}(t,x,\tau)\right|
\,dx\,dt \\
&\le
\frac{1}{\tau}
\int_{\Omega_T}
\left|a(x)\varphi(x+t\omega)
\left(
\omega\cdot\nabla\varphi(x+t\omega)A^{+}(t,x)
\right)
\right|
\,dx\,dt \\
&\quad
+ \frac{1}{\tau}
\int_{\Omega_T}
\left|a(x)\varphi^2(x+t\omega)\partial_t A^{+}(t,x)\right|
\,dx\,dt \\
&\quad
+ \frac{1}{\tau}
\int_{\Omega_T}
\left|a(x)\varphi(x+t\omega)\partial_t R_{k}(t,x,\tau)\right|
\,dx\,dt \lesssim
\frac{1}{\tau}\|\varphi\|^{2}_{H^{3}(\R^n)}.
\end{align*}

\noindent\textbf{Estimate of $I_4$.}
Similarly,
\begin{align*}
|I_4|
&\le
\frac{1}{\tau}
\int_{\Omega_T}
\left|a(x)R_{11}(t,x,\tau)R_0(t,x,\tau)\right|
\,dx\,dt \\
&\lesssim
\frac{1}{\tau}
\left(
\|\nabla\varphi\|_{L^2(\R^n)}
\|R\|_{L^2(\Omega_T)}
+
\|R\|_{L^2(\Omega_T)}
+
\|\partial_t R_k\|_{L^2(\Omega_T)}
\|R\|_{L^2(\Omega_T)}
\right) \\
&\lesssim
\frac{1}{\tau}\left(\|\varphi\|^{2}_{H^{3}(\R^n)}+ \|\varphi\|_{H^{3}(\R^n)}\right).
\end{align*}

\noindent\textbf{Estimate of $I_5$.}
\begin{align*}
|I_5|
=
\frac{1}{\tau}
\int_{\Omega_T}
\left|b(x)\varphi^2(x+t\omega)A^{+}(t,x)\right|
\,dx\,dt \lesssim
\frac{1}{\tau}\|\varphi\|^{2}_{H^{3}(\R^n)}.
\end{align*}

\noindent\textbf{Estimate of $I_6$.}
\begin{align*}
|I_6| &= \frac{1}{ \tau}\int_{\Omega_T} \left|b(x)\varphi(x+t\omega)A^{+}(t,x)e^{i\tau(x\cdot\omega + t)}R_0(t,x,\tau)\right|\,dx\,dt\\
&\lesssim
\frac{1}{\tau}\|\varphi\|_{L^2(\R^n)}
\|R\|_{L^2(\Omega_T)} \lesssim
    \frac{1}{\tau^2}\|\varphi\|^{2}_{H^{3}(\R^n)}.
\end{align*}

\noindent\textbf{Estimate of $I_7$.}
\begin{align*}
|I_7|
\lesssim
\frac{1}{\tau}
\|\varphi\|_{L^2(\R^n)}
\|R_k\|_{L^2(\Omega_T)} \lesssim
\frac{1}{\tau^2}\|\varphi\|^{2}_{H^{3}(\R^n)}.
\end{align*}

\noindent\textbf{Estimate of $I_8$.}
\begin{align*}
|I_8|
\lesssim
\frac{1}{\tau}
\|R_k\|_{L^2(\Omega_T)}
\|R\|_{L^2(\Omega_T)} \lesssim
\frac{1}{\tau^3}\|\varphi\|_{H^{3}(\R^n)}^2.
\end{align*}
Combining the above estimates for each of the integrals
$I_2,\dots,I_8$ in~\eqref{eq:final-splitting} and letting
$\tau \to \infty$, we conclude that all contributions of these integrals vanish. Hence, after taking $\tau\rightarrow \infty$ in~\eqref{eq:final-splitting}, we get 
\begin{align}\label{eq:main-identity}
\int_{\Omega_T}
a(x)\,\varphi^{2}(x+t\omega)\,A^{+}(t,x)\,dx\,dt
= \text{known},\ \mbox{for all $\varphi\in \mathfrak{C}_c^{\infty}(\mathbb{R}^n)$ and $\omega\in\mathbb{S}^{n-1}$}. 
\end{align} 
Since $a\in C_c^{\infty}(\Omega)$, we extend $a$ by $0$ outside of $\Omega$, so we have
\begin{align*}
\int_0^T\int_{\mathbb{R}^n} a(x)\varphi^2(x+t\omega)A^{+}(t,x)\,dx\,dt = \text{known}, \ \mbox{for all $\varphi\in \mathfrak{C}_c^{\infty}(\mathbb{R}^n)$ and $\omega\in\mathbb{S}^{n-1}$}.
\end{align*}
Now after using the  change of variable  $ x +t\omega:=y$, along with the Fubini's Theorem, we get 
\begin{align}\label{eq; distribution = known}
\int_{\mathbb{R}^n}\left(\int_0^T a(y -t\omega)A^{+}(t,y -t\omega)\,dt\right)\varphi^2(y)\,dy = \text{known},\ \mbox{for all $\varphi\in \mathfrak{C}_c^{\infty}(\mathbb{R}^n)$ and $\omega\in\mathbb{S}^{n-1}$}. 
\end{align}
 After defining  $ I_{a}(y,\omega) := \int_0^T a(y -t\omega)A^{+}(t,y -t\omega)\,dt$, for $(y,\omega)\in \mathbb{R}^n\times \mathbb{S}^{n-1}$, we obtain that  
\begin{align}\label{T operator is known}
 \int_{\mathbb{R}^n}I_{a}(y,\omega)\varphi^2(y)\,dy=\mbox{known},\  \mbox{for all $\varphi\in \mathfrak{C}_c^{\infty}(\mathbb{R}^n)$ and $\omega\in\mathbb{S}^{n-1}$}.
\end{align}
Next,  for $r \in \left(0, \min\left\{1, \frac{T - \operatorname{diam}(\Omega)}{3}\right\}\right)$, we introduce an open set $\Omega_{r}\subset\mathbb{R}^n$ by 
\begin{align}\label{dom:Omega_epsilon}
    \Omega_{r}
    := \left\{ x \in \mathbb{R}^n \setminus \overline{\Omega} : 
    \operatorname{dist}(x, \Omega) < r \right\}. 
\end{align}
Also for   $\chi\in C_{c}^{\infty}(B_1(0))$ such that  $0\leq \chi \leq 1$ and $\int_{\mathbb{R}^n}\chi^2(x)dx=1$ and  $y_0\in \Omega_{r}$, we consider a  family of test functions $\{\chi_{\eta}\}_{\eta>0}\subset C_c^{\infty}(\R^n)$ by 
\begin{align}\label{Phi_{delta}}
    \chi_{\eta}(y)=\eta^{-\frac{n}{2}}\chi\left(\frac{y-y_0}{\eta}\right).
\end{align}
 By the construction of $\{\chi_{\eta}\}_{\eta>0}$ and $\Omega_{r}$, we  observe that $\{\chi_{\eta}\}_{\eta>0}\subset \mathfrak{C}_{c}^{\infty}(\mathbb{R}^{n})$ whenever $0<\eta<r$. 
Thus, using this family $\{\chi_{\eta}\}_{\eta>0}$ of test functions in Equation \eqref{T operator is known}, we get that
\begin{align*}
 \text{known} &= \int_{\mathbb{R}^n}I_{a}(y,\omega)\chi^2_{\eta}(y)\,dy \\
 &=\eta^{-n}\int_{B(y_0,\eta)}I_{a}(y,\omega)\chi^2\left(\frac{y-y_0}{\eta}\right)\,dy,\ \ \mbox{for all}\ 0<\eta<r\ \mbox{and}\ \ y_0\in \Omega_{r}. 
\end{align*}
Now taking $\eta\rightarrow 0$, in the above identity, we arrive at
\begin{align}\label{Ia is known}
    I_{a}(y_0,\omega)=\mbox{known},\ \mbox{for any $y_0\in \Omega_{r}$ and $\omega\in\mathbb{S}^{n-1}$}.
\end{align}
Next, we simplify the expression of $I_{a}(y_0,\omega)$ as follows
\begin{align*}
   I_{a}(y_0,\omega) &:= \int_0^T a(y_{0} -t\omega)A^{+}(t,y_{0} -t\omega)\,dt \ = \int_0^T a(y_{0} -t\omega)\exp\left(-\frac{1}{2}\int_0^ta(y_{0}+(s-t)\omega)\,ds\right)\,dt\\
    & = \int_0^T -2\,\partial_t \exp\left(-\frac{1}{2}\int_0^ta(y_{0}-\rho\omega)\,d\rho\right)\,dt= -2 \left(\exp\left(-\frac{1}{2}\int_0^T a(y_{0}-\rho\omega)\,d\rho\right) - 1\right), 
\end{align*}
where to arrive at the final expression, we have used the substitution $\rho=t-s$. 
Using this in Equation \eqref{Ia is known}, we get   
\begin{align}\label{Ia known for finite time}
   \int_0^T a(y_{0}-\rho\omega)\,d\rho = \text{known},\  \mbox{for all $y_0\in \Omega_{r}$ and $\omega\in \mathbb{S}^{n-1}$}. 
\end{align}
Now since $a$ is compactly supported in $\Omega$,  $(\Omega_{r} \pm T'\omega)\cap \Omega = \emptyset\text{ for all }T'\geq T$ and $\omega\in \mathbb{S}^{n-1}$, therefore
$a(y_{0}-\rho\omega)=0$ for $\rho\geq T$. Combining all these, yields that 
\begin{align}\label{estimate a_{omega_{eps}}}
   \int_{0}^{\infty} a(y_{0}-\rho\omega)\,d\rho = \text{known}, \ \text{ for all } y_{0}\in \Omega_{r}\ \mbox{and}\ \omega\in \mathbb{S}^{n-1}.
\end{align}
In order to reconstruct that the $X-$ray transform of $a$ is known, we must show that the left-hand side of the above integral is known for any $y_{0}\in \R^n$. To do so, we start with a key  observation which states  that if the line through a point $y_{0}\in \mathbb{R}^{n}$ 
in the direction $\omega$ intersects the domain $\Omega_{r}$, 
then the integral \eqref{estimate a_{omega_{eps}}} follows immediately 
from a change of variable formula. 
We first show that
\eqref{estimate a_{omega_{eps}}} can be shown to hold for every 
$y_{0}\in \overline{\Omega}$ which amounts into knowing $\int_{0}^{\infty} a(y_{0}-\rho\omega)\,d\rho$, for all $y_{0}\in \Omega_{r}\cup \overline{\Omega}$. By the construction of $\Omega_{r}$, 
any point in $\overline{\Omega}$ can be shifted along the direction $\omega\in\mathbb{S}^{n-1}$ 
into $\Omega_{r}$; that is,
\begin{align*}
     y_{0}\in \overline{\Omega}
\ \ \mbox{if and only if} \ \ 
    y_{0}+t_{0}\omega \in \Omega_{r}, \ \ \text{for some } t_{0}\in \mathbb{R}.
\end{align*}
Hence, using the previous observation, we have that the integral $\int_{0}^{\infty} a(y_{0}-\rho\omega)\,d\rho$ is known for all $y_{0}\in \Omega_{r}\cup \overline{\Omega}$ and $\omega\in \mathbb{S}^{n-1}$.

Next, we consider the case when $y_{0}\in (\Omega_{r}\cup \overline{\Omega})^{c}$. 
Now if the line $y_{0}+t\omega$ intersects $\Omega_{r}$, then, as before, 
the result holds by repeating the previous step and if the line through $y_{0}$ does not intersect $\Omega_{r}$, 
then it also does not intersect $\Omega$ as well. Hence, using the fact that $a\in C_{c}^{\infty}(\Omega)$ and combing the above two cases,  we conclude that either $\int_{0}^{\infty} a(y_{0}-\rho\omega)\,d\rho=0$ or known for each $y_{0}\in (\Omega_{r}\cup \overline{\Omega})^{c}$ and $\omega\in\mathbb{S}^{n-1}$. 
 Thus, we obtain that 
 \begin{align*}
      \int_{0}^{\infty} a(y_{0}-\rho\omega)\,d\rho = \text{known}, \ \text{ for all } y_{0}\in \mathbb{R}^n\ \mbox{and}\ \omega\in \mathbb{S}^{n-1},
 \end{align*}
 this immediately implies that $X-$ray  transform $X_a(\omega,y)$ of  $a$ at $y\in\mathbb{R}^n$, in the  direction of  $\omega\in\mathbb{S}^{n-1}$ defined below satisfies 
\begin{align}\label{estimate a_{whole domain}}
 X_a(\omega,y)  := \int_{\R} a(y+\rho\omega)\,d\rho = \text{known}, \quad \mbox{for all}\ (\omega,y)\in T\mathbb{S}^{n-1},
\end{align}
where $ T\mathbb{S}^{n-1} := \{\, (\omega,y) \,|\, \omega\in \mathbb{S}^{n-1}, y \in \omega^{\perp}\}$ is the tangent bundle of the unit sphere $\mathbb{S}^{n-1}$ in $\mathbb{R}^n$. Now using  the Fourier slice Theorem for $X-$ray transform (see Theorem 1.1, \cite{Natterer86}), we get that 
\begin{align}
    \widehat{a}(\xi)
    = \widehat{X_a}(\omega,\xi)
    = \text{known},
    \quad \text{for each } \xi \in \omega^{\perp} \ \mbox{and}\ \omega\in \mathbb{S}^{n-1}.
\end{align}
Now since $\displaystyle \cup_{\omega\in \mathbb{S}^{n-1}}\omega^{\perp}=\mathbb{R}^n$, therefore 
 $\widehat{a}(\xi)$ is known for all $\xi \in \mathbb{R}^n$.
Finally, using  the Fourier inversion formula yields  that the function $a$ is uniquely determined in $\mathbb{R}^n$. Finally, because the function $a$ is compactly supported in $\Omega$, we conclude that
\begin{align}
    a(x) = \text{known }, \quad \text{for all  } x \in \Omega.
\end{align}
This completes the reconstruction of the damping coefficient $a$ in $\Omega$.\qed

\vspace{.4cm}

Next, we reconstruct the linear potential $b$. Since we have already reconstructed the damping coefficient $a$, therefore, $v_{a}$  the solution to the following backward wave problem
\begin{align}\label{FBVP; widetilde{v_0}}
\begin{cases}
\Box v_{a}(t,x) - a(x)\,\partial_t v_{a}(t,x)= 0, & (t,x)\in \Omega_T, \\
v_{a}(T,x) = 0, \quad \partial_t v_{a}(T,x) = 0, & \quad x \in \Omega
\end{cases}
\end{align}
is known in $\Omega_T$. 
Now, we multiply Equation \eqref{eq;Notations of first order Linearization eqn}  by $v_{a}$, and integrate over $\Omega_T$, to obtain 
\begin{align}
    \int_{\Omega_T} \left(\Box\, v_k(t,x) + a(x)\partial_t\,v_k(t,x) + b (x)v_k(t,x)\right)v_{a}(t,x)\,dx\,dt = 0.
\end{align}
Using the  integration by parts, we have
\begin{align}\label{eq;  b(x) in terms of known data}
\begin{aligned}
&-\int_{\Omega_T}b (x)v_k(t,x)v_{a}(t,x)\,dx\,dt = \int_{\Omega_T}  v_{k}(t,x)  \left(\Box v_{a}(t,x) - a(x)\,\partial_t v_{a}(t,x)\right)\,dx\,dt\\
& \qquad +\int_{\Omega}\big[(\partial_t v_{k}v_{a}- \partial_t v_{a}v_k)(T,x)\big]\,dx- \int_{\Omega}\big[(\partial_t v_{k}v_{a}- \partial_t v_{a}v_k)(0,x)\big]\,dx \\
& \qquad + \int_{\Omega}a(x)\big[ v_{k}(T,x)v_{a}(T,x)-  v_{a}(0,x)v_k(0,x)\big]\,dx - \int_{\Sigma}v_{a}(t,x)\partial_{\nu} v_k(t,x)\,dS_x\,dt \\
& \qquad + \int_{\Sigma} v_k(t,x)\partial_{\nu} v_{a}(t,x) \,dS_x\,dt.   
\end{aligned}
\end{align}
On the right-hand side of Equation \eqref{eq;  b(x) in terms of known data}, the first four integrals vanish using \eqref{eq;Notations of first order Linearization eqn}, \eqref{FBVP; v}, \eqref{FBVP; widetilde{v_0}} and the last two integrals are known from the DN map \eqref{eq:DN-map-v_k}, hence, in view of this, we arrive at 
\begin{align}\label{integral identity b = known}
\int_{\Omega_T} b (x)v_k(t,x)v_{a}(t,x)\,dx\,dt = \text{known}, 
\end{align}
for all $v_k$ ($1\leq k\leq \ell$) and $v_{a}$ solving \eqref{eq;first order Linearization eqn} and \eqref{FBVP; widetilde{v_0}} respectively.
 Next after substituting the special solutions for $v_{k}$ and $v_{a}$ from \eqref{eq: v_{(1)}} and \eqref{eq: v} in the above mentioned  integral identity  \eqref{integral identity b = known},  we conclude that 
\begin{align}
 \text{known} =
 \int_{\Omega_T} 
&\Big\{
    b(x)\left(
        \varphi(x+t\omega)\, A^{+}(t,x)\, e^{i\tau(x\cdot \omega + t)}
        + R_{1}(t,x,\tau)
    \right)\times\\
   & \left(\varphi(x+t\omega)\, A^{-}(t,x)\,e^{-i\tau(x\cdot \omega + t)}
    + R_{2}(t,x,\tau)\right)
\Big\}
\, dx\, dt,\ \mbox{for all $\varphi\in \mathfrak{C}_c^{\infty}(\mathbb{R}^n)$ and $\omega\in\mathbb{S}^{n-1}$}.
\end{align}
Simplifying the above expression 
\begin{align}\label{eq:final-splitting in b}
\begin{aligned}
   \text{known} =  &\int_{\Omega_T} b(x) \varphi^2(x+t\omega)\,dx\,dt + \int_{\Omega_T}b(x)\varphi(x+t\omega)A^{-}(t,x)R_{1}(t,x,\tau)e^{-i\tau(x\cdot\omega + t)}\,dx\,dt \\&+ \int_{\Omega_{T}}b(x)\varphi(x+t\omega)A^{+}(t,x)e^{i\tau(x\cdot\omega + t)}R_{2}(t,x,\tau) \,dx\,dt+\int_{\Omega_{T}}b(x)R_{1}(t,x,\tau)R_{2}(t,x,\tau)\,dx\,dt,\\
    &=\int_{\Omega_T} b(x) \varphi^2(x+t\omega)\,dx\,dt + J_2 +J_3  +J_4,\ \mbox{for all $\varphi\in \mathfrak{C}_c^{\infty}(\mathbb{R}^n)$ and $\omega\in\mathbb{S}^{n-1}$}. 
    \end{aligned} 
\end{align}
The integrals $J_k$, $k=2,3,4$, can be estimated in an analogous manner to
$I_j$, $j=2,\dots,8$, using H\"older's inequality together with the
estimates~\eqref{eq: r_{(1)}} and~\eqref{eq: estimate of r_{(2)}}.
In particular, each $J_k$ admits a bound of order $\mathcal{O}(\tau^{-1})$, so each $J_k$ terms vanish as $\tau \to \infty$.
Thus, letting $\tau \to \infty$ in the Equation \eqref{eq:final-splitting in b}, we get the following identity
\begin{align}
    \int_{\Omega_T} b(x) \varphi^2(x+t\omega)\,dx\,dt = \text{known}, \ \mbox{for all $\varphi\in \mathfrak{C}_c^{\infty}(\mathbb{R}^n)$ and $\omega\in\mathbb{S}^{n-1}$}.
\end{align} 
Now since $b\in C_c^{\infty}(\Omega)$, therefore after extending  $b$ by $0$ outside of $\Omega$, we obtain that 
\begin{align*}
\int_0^T\int_{\mathbb{R}^n} b(x)\varphi^2(x+t\omega)\,dx\,dt = \text{known}, \  \mbox{for all $\varphi\in \mathfrak{C}_c^{\infty}(\mathbb{R}^n)$ and $\omega\in\mathbb{S}^{n-1}$}.
\end{align*}
Now, after using the change of variable formula  $ x +t\omega:=y$, along with an application of  Fubini's Theorem, gives us  that 
\begin{align}
\int_{\mathbb{R}^n}\left(\int_0^T b(y -t\omega)\,dt\right)\varphi^2(y)\,dy = \text{known},\  \mbox{for all $\varphi\in \mathfrak{C}_c^{\infty}(\mathbb{R}^n)$ and $\omega\in\mathbb{S}^{n-1}$}.
\end{align} 
Finally, repeating arguments used in the reconstruction of the damping term $a$ from Equation \eqref{Ia known for finite time} to the reconstruction of $a$, we conclude that $
    b(x) \  \text{is known for all}\  x \in \Omega.$
This completes the reconstruction of the linear potential $b$ in $\Omega$.\qed
 
\subsection{Reconstruction of the nonlinear coefficient}

 As mentioned above, in the present subsection, we make use of reconstructed coefficients associated with the linear term along with the higher-order
linearization of the DN map, to reconstruct the nonlinear potential $q$ in $\Omega$. 
Towards this objective, first, we recall the following 
Fa\`a di Bruno formula (see \cite{di1857note,krantz2002primer})  related to the chain rule for higher order derivatives will be instrumental in our future analysis. If $f$ and $g$ are two smooth  functions defined on an open set in $\mathbb{R}^n$ are such that $f\circ g$ is well defined for each point in domain of $g$,  then the following formula holds for any
$n \in \mathbb{N}$
\begin{align}\label{eq:faadibruno}
\begin{aligned}
(f\circ g)^n(x)
= \sum_{\pi \in \mathcal{P}(\{1,\dots,n\})}
f^{(|\pi|)}\!\big(y\big)
\prod_{B \in \pi} \frac{\partial^{(|B|)}g(x)}{\prod_{j \in B} \partial x_j}, \ \mbox{for}\  x\in\mbox{domain}(g)\ \mbox{and} \ y:=g(x),
\end{aligned} 
\end{align} 
where $\mathcal{P}(\{1,\dots,n\})$ denotes the set of all partitions of the index set $\{1,\dots,n\}$ which means that each partition $\pi \in \mathcal{P}(\{1,\dots,n\})$ consists of pairwise disjoint nonempty subsets of $\{1,\dots,n\}$, called \emph{blocks}, whose union is exactly $\{1,\dots,n\}$. We write $|\pi|$ for the number of blocks in the partition $\pi$ for each block $B\in\pi$, and $|B|$ represents its cardinality.

Now apply the formula \eqref{eq:faadibruno} to the composition $f\circ g$, where $f(z):=z^\ell$ and $g(\epsilon) := u_{\eps f}$.  Then, the mixed ${\ell}$-th derivative of $u_{\eps f}^{\ell}$ with respect to $\epsilon_1,\dots,\epsilon_\ell$, is given by 
\begin{align}\label{eq:first_db-step}
    \frac{\partial^\ell \big(f(g(\epsilon))\big)}{\partial \epsilon_1\cdots\partial \epsilon_\ell}
= \sum_{\pi\in \mathcal{P}(\{1,\dots,\ell\})}
f^{(|\pi|)}\!\big(g(\epsilon)\big)\,
\prod_{B\in\pi}
\frac{\partial^{|B|} g(\epsilon)}{\prod_{j\in B}\partial \epsilon_j}.
\end{align}
Next, compute the derivatives of $f(z)=z^\ell$. For $k=1,\dots,\ell$,
\begin{align}
    f^{(k)}(z)=\ell(\ell-1)\cdots(\ell-k+1)\,z^{\ell-k}
=\frac{\ell!}{(\ell-k)!}\,z^{\ell-k},
\end{align}
and $f^{(k)}(z)=0,$ for $k>\ell$. Therefore, only partitions $\pi$ with
$|\pi|=k\le \ell$ contribute. Now, grouping the sum in \eqref{eq:first_db-step} according to the number of blocks $k=|\pi|$ yields that 
\begin{align*} 
&    \frac{\partial^\ell \big(f(g(\epsilon))\big)}{\partial \epsilon_1\cdots\partial \epsilon_\ell}
= \frac{\partial^{\ell} u_{\eps f}^{\ell}}{\partial{\eps_1}\dots\partial{\eps_{\ell}}} = \sum_{\pi\in \mathcal{P}(\{1,\dots,\ell\})} \frac{{\ell}!}{({\ell}-k)!}\, u_{\eps f}^{\,{\ell}-k}
\ \prod_{B \in \pi}\frac{\partial^{|B|} u_{\eps f}}
{\prod_{j \in B} \partial \eps_j}\\
& \quad = \sum_{k=1}^{\ell} \sum_{\pi \in \mathcal{P}_k(\{1,\dots,{\ell}\})} \frac{{\ell}!}{({\ell}-k)!}\, u_{\eps f}^{\,{\ell}-k}
\ \prod_{B \in \pi}\frac{\partial^{|B|} u_{\eps f}}
{\prod_{j \in B} \partial \eps_j}= \sum_{k=1}^{\ell} \frac{{\ell}!}{({\ell}-k)!}\, u_{\eps f}^{\,{\ell}-k}
\ \sum_{\pi \in \mathcal{P}_k(\{1,\dots,{\ell}\})} \prod_{B \in \pi}\frac{\partial^{|B|} u_{\eps f}}
{\prod_{j \in B} \partial \eps_j}.
\end{align*}
Thus, we have obtained the following version of the  Fa\`a di Bruno formula, which will be used in 
reconstruction of the nonlinear coefficient.
\begin{align}\label{eq:final_db-step}
    \frac{\partial^{\ell} u_{\eps f}^{\ell}}{\partial{\eps_1}\dots\partial{\eps_{\ell}}} = \sum_{k=1}^{\ell} \frac{{\ell}!}{({\ell}-k)!}\, u_{\eps f}^{\,{\ell}-k}
\sum_{\pi \in \mathcal{P}_k(\{1,\dots,{\ell}\})}
\ \prod_{B \in \pi}\frac{\partial^{|B|} u_{\eps f}}
{\prod_{j \in B} \partial \eps_j},
\end{align}
where $\mathcal{P}_k(\{1,\dots,\ell\})$ is the set of partitions of the set $\{1,\dots,\ell\}$ into $k$ nonempty blocks $B$.   Next, evaluating the above formula \eqref{eq:final_db-step} at  $\epsilon=0$, along with using the fact that $u_{\eps f}(t,x)\big|_{\epsilon=0}=0,$ in $\Omega_T$, whenever $u_{\epsilon f}$ is solution to \eqref{eq;Linearization}, we get that  
\begin{align}\label{eq: evaluate at eps=0}
     \frac{\partial^{\ell} u_{\eps f}^{\ell}}{\partial{\eps_1}\dots\partial{\eps_{\ell}}}\Biggl|_{\eps = 0} = {\ell}! \prod_{k=1}^{\ell}
\left.  \frac{\partial u_{\eps f}}{\partial{\eps_k}} \right|_{\epsilon=0}.
\end{align}
Now, after evaluating the mixed ${\ell}$-th order derivatives with respect to the parameters
$\epsilon_1,\dots,\epsilon_{\ell}$ of  equations in 
\eqref{eq;Linearization} at $\epsilon=0$, along  with   using \eqref{eq: evaluate at eps=0} and $v_k:=\left. \frac{\partial u_{\eps f}}{\partial{\eps_k}}\right|_{\epsilon=0}:=\left. \partial_{\eps_k} u_{\eps f} \right|_{\epsilon=0}$,  for $1\leq k\leq \ell$, we obtain that $w:= {\partial^{\ell}_{\epsilon_1\dots\epsilon_{\ell}}}{\big|_{\epsilon = 0} }u_{\epsilon f}$,  satisfies
 \begin{align}\label{eq; Non hom linear IBVP}
     \begin{cases}
         \Box w(t,x) + a(x)\partial_t w(t,x) + b(x) w(t,x) = -{\ell}! q(x) \prod\limits_{k=1}^{\ell}v_k(t,x), &(t,x) \in \Omega_T,\\
    w(t,x)   = 0, & (t,x) \in\Sigma, \\
     w (0, x )=0,~~ \partial_t w (0, x ) = 0,  & \quad x \in\Omega,
     \end{cases}
 \end{align}
 where $v_k$ for $1\leq k\leq \ell$, are solutions to \eqref{eq;Linearization}, as  established in previous subsection. 
Now since the DN map $\Lambda_{a,b,q}(\epsilon f)$  (see~\eqref{eq:DN map}) is assumed to be
known for each $\epsilon f$, therefore using 
\begin{align}\label{eq:DN-map-w}
\left[\left.
\partial_{\epsilon_1\dots\epsilon_{\ell}}^{\ell}
\Lambda_{a,b,q}(\epsilon_1 f_1 + \dots+\epsilon_{\ell} f_{\ell})\big|_{\Sigma}\right]
\right|_{\epsilon=0}
=
\left[\left.
\partial_{\epsilon_1\dots\epsilon_{\ell}}^{\ell}
\partial_\nu u_{\epsilon f}
\big|_{\Sigma}\right]
\right|_{\epsilon=0}
=
\partial_\nu w \big|_{\Sigma},
\end{align}
we get that the normal derivative $\partial_\nu w$ is also known on $\Sigma$.

 Next, our aim is to establish an integral identity which relates the unknown $q$ to that of the known data. To obtain this, we  multiply the first equation in  IBVP \eqref{eq; Non hom linear IBVP} by auxiliary function $w_{0}$, solving the following backward  problem
\begin{align}\label{FBVP; w}
     \begin{cases}
         \Box w_0(t,x) - a(x)\,\partial_t w_0(t,x) + b(x)\, w_0(t,x)  = 0, &(t,x) \in \Omega_T,\\
     w_0 (T,x)=0,~~ \partial_t w_0 (T,x) = 0,  & \quad x \in\Omega,
     \end{cases}
 \end{align}
 and integrate over $\Omega_T$, to obtain
  \begin{align}
    \int_{\Omega_T} \left(\Box w(t,x) + a(x)\,\partial_t w(t,x) + b(x)\, w(t,x) +{\ell}! q(x)\prod\limits_{k=1}^{\ell}v_k(t,x)\right)w_{0}(t,x)\,dx\,dt = 0.
\end{align}
Now, using integration by parts, we have 
\begin{align}\label{eq; a(x), b(x) and q(x) in terms of known data}
\begin{aligned} 
& \int_{\Omega_T} {\ell}! q(x)\left(\prod\limits_{k=1}^{\ell}v_k(t,x)\right)\, w_0(t,x)\,dx\,dt =\\& \qquad - \int_{\Omega_T} w(t,x) \left(\Box\, w_0(t,x) - a(x)\,\partial_t w_0(t,x) + b(x)\, w_0(t,x)\right)\,dx\,dt \\ & \qquad  
-\int_{\Omega}\big[ w_0 \partial_t w- w \partial_t w_0 \big](T,x)\,dx  + \int_{\Omega}\big[w_0 \partial_t w  - w \partial_t w_0 \big](0,x)\,dx \\ & \qquad - \int_{\Omega}a(x)\big[ w(T,x)w_0(T,x)-  w(0,x)w_0(0,x)\big]\,dx + \int_{\Sigma}w_0(t,x)\partial_{\nu} w(t,x)\,dS_x\,dt\\& \qquad -\int_{\Sigma} w(t,x)\partial_{\nu} w_0(t,x)\,dS_x\,dt. 
\end{aligned} 
\end{align}
On the right-hand side of Equation \eqref{eq; a(x), b(x) and q(x) in terms of known data}, the first four integrals vanish using   \eqref{eq; Non hom linear IBVP} and \eqref{FBVP; w}. Also, using the knowledge of the DN map $\Lambda_{a,b,q}$ along with \eqref{eq:DN-map-w}, we conclude that the last two terms are known. Hence, we arrive at 
\begin{align}\label{Inegral_Identity_1}
 \int_{\Omega_T}  q(x)\, \left(\prod\limits_{k=1}^{\ell}v_k(t,x)\right)\, w_0(t,x)\,dx\,dt = \text{known}, 
\end{align}
for any choices of $v_k \ (1\leq k\leq\ell)$, solving \eqref{eq;subfirst order Linearization eqn} and $w_0$ solving \eqref{FBVP; w}. 
Next we  use the  asymptotic solutions for $w_0$ and $v_k$ $(1\le k\le \ell)$ from  Propositions \ref{Asymptotic solutions} and \ref{Asymptotic solutions for IBVP},  given by 
\begin{equation}
 w_{0}(t,x)= e^{-{\ell}i\tau\,(t+x \cdot\omega)}\left(\varphi(x+ t\omega)
\exp\left(\int_{0}^{t}\frac{a(x+\tau\omega)}{2}\,d\tau\right)+\sum\limits_{i=1}^{N}\frac{\beta_{i}(t,x)}{\tau^i}\right)+R_0(t,x,\tau),
    \end{equation}
and 
\begin{equation}
 v_{k}(t,x)= e^{i\tau\,(t+x \cdot\omega)}\left(\varphi(x+ t\omega)
\exp\left(-\int_{0}^{t}\frac{a(x+\tau\omega)}{2}\,d\tau\right)+\sum\limits_{i=1}^{N}\frac{m_{i}(t,x)}{\tau^i}\right)+R_{k}(t,x,\tau), 
    \end{equation}
   respectively,  where $\omega\in \mathbb{S}^{n-1}$, $\varphi\in \mathfrak{C}^{\infty}_{c}(\mathbb{R}^n)$,  $R_{j}(t,x,\tau)$ for $0\leq j\leq \ell$, satisfy the estimate $\lVert R_j\rVert_{L^{\infty}(\Omega_T)}\leq \frac{C}{\tau}$, for some constant $C>0$ independent of $\tau$ along with the following 
    \[R_0(T,x,\tau)=\partial_{t}R_{0}(T,x,\tau)=0,\  \mbox{and}\  R_k(0,x,\tau)=\partial_t R_k(0,x,\tau)=0,\ (1\leq k\leq \ell),\   \mbox{for}\ x\in \Omega.\] 
Now, for the sake of  convenience, we write
 \begin{align}
     w_{0} = \mathcal{A}^{-\ell}(\mathcal{F}+\mathcal{G}) +R_0, \qquad v_{k} = \mathcal{A}(\mathcal{P}+\mathcal{Q}) +R_{k}, \qquad 1\le k\le \ell,
 \end{align}
where
\begin{align}\label{eq; a few notations after IE}
\begin{aligned}
   & \mathcal{A} := e^{i\tau\,(t+x \cdot\omega)}, \quad \mathcal{F}:= \varphi(x+ t\omega)
\exp\left(\int_{0}^{t}\frac{a(x+\tau\omega)}{2}\,d\tau\right), \quad \mathcal{G} := \sum\limits_{i=1}^{N}\frac{\beta_{i}(t,x)}{\tau^i},\\
& \mathcal{P}:= \varphi(x+ t\omega)
\exp\left(-\int_{0}^{t}\frac{a(x+\tau\omega)}{2}\,d\tau\right),\quad
\mathcal{Q} := \sum\limits_{i=1}^{N}\frac{m_{i}(t,x)}{\tau^i}.
\end{aligned}
\end{align}
Using the above notations, we now compute the term $w_0\!\left(\prod_{k=1}^{\ell}v_k\right)$,  appearing in the integral identity as follows 
\begin{align}
 w_0\!\left(\prod_{k=1}^{\ell}v_k\right)
 = \left( \mathcal{A}^{-\ell}(\mathcal{F}+\mathcal{G}) + R_0\right)
 \left(\prod_{k=1}^{\ell}\big(\mathcal{X}+R_{k}\big)\right),
\end{align}
where $\mathcal{X} := \mathcal{A}(\mathcal{P}+\mathcal{Q})$. Using the identity
\begin{align}
\prod_{k=1}^{\ell}\big(\mathcal{X}+R_{k}\big)
= \sum_{J\subset\{1,\dots,\ell\}}
\mathcal{X}^{\ell-|J|}\prod_{j\in J} R_{j},
\end{align}
we obtain that 
\begin{align}
\begin{aligned}
w_0\!\left(\prod_{k=1}^{\ell}v_k\right)
&= \left( \mathcal{A}^{-\ell}(\mathcal{F}+\mathcal{G}) + r\right)
\sum_{J\subset\{1,\dots,\ell\}}
\mathcal{X}^{\ell-|J|}\prod_{j\in J} R_{j} \nonumber\\[2pt]
&= \mathcal{A}^{-\ell}(\mathcal{F}+\mathcal{G})(\mathcal{A}(\mathcal{P}+\mathcal{Q}))^{\ell}  + \mathcal{A}^{-\ell}(\mathcal{F}+\mathcal{G})
\sum_{\emptyset\neq J\subset\{1,\dots,\ell\}}
(\mathcal{A}(\mathcal{P}+\mathcal{Q}))^{\ell-|J|}\prod_{j\in J} R_{j} \nonumber\\[2pt]
&\quad + R_0\sum_{J\subset\{1,\dots,\ell\}}
(\mathcal{A}(\mathcal{P}+\mathcal{Q}))^{\ell-|J|}\prod_{j\in J} R_{j}.
    \end{aligned}
\end{align}
Since $\mathcal{A}^{-\ell}(\mathcal{A}(\mathcal{P}+\mathcal{Q}))^{\ell}=(\mathcal{P}+\mathcal{Q})^{\ell}$, it follows that
\begin{align}
\begin{aligned}
w_0\!\left(\prod_{k=1}^{\ell}v_k\right)
&= (\mathcal{F}+\mathcal{G})(\mathcal{P}+\mathcal{Q})^{\ell}  + \mathcal{A}^{-\ell}(\mathcal{F}+\mathcal{G})
\sum_{\emptyset\neq J\subset\{1,\dots,\ell\}}
(\mathcal{A}(\mathcal{P}+\mathcal{Q}))^{\ell-|J|}\prod_{j\in J} R_{j}\\ 
&\qquad + R_0\sum_{J\subset\{1,\dots,\ell\}}
(\mathcal{A}(\mathcal{P}+\mathcal{Q}))^{\ell-|J|}\prod_{j\in J} R_{j}.
\end{aligned} 
\end{align}
Expanding of  $(\mathcal{P}+\mathcal{Q})^{\ell}$ using Binomial Theorem, yields that 
\begin{align}
\begin{aligned}
w_0\!\left(\prod_{k=1}^{\ell}v_k\right)
&= \mathcal{F} \mathcal{P}^{\ell} + \mathcal{G} \mathcal{P}^{\ell}
+ (\mathcal{F}+\mathcal{G})\sum_{j=1}^{\ell} \mathcal{P}^{\ell-j}\mathcal{Q}^{j}  + \mathcal{A}^{-\ell}(\mathcal{F}+\mathcal{G})
\sum_{\emptyset\neq J\subset\{1,\dots,\ell\}}
(\mathcal{A}(\mathcal{P}+\mathcal{Q}))^{\ell-|J|}\prod_{j\in J} R_{j}\\ &\qquad  + R_0 \sum_{J\subset\{1,\dots,\ell\}}
(\mathcal{A}(\mathcal{P}+\mathcal{Q}))^{\ell-|J|}\prod_{j\in J} R_{j}.
\end{aligned} 
\end{align}
After using the above expansion along with equation \eqref{eq; a few notations after IE} in the integral identity \eqref{Inegral_Identity_1}, we obtain
\begin{align}\label{eq:final-splitting in q}
\begin{aligned}
\int_{\Omega_T}
q(x)\,\varphi^{\ell+1}(x+t\omega)\,\left(A^{+}(t,x)\right)^{\ell -1}
\,dx\,dt
+ \sum_{j=2}^{5}\widetilde{I}_j=\;& \text{known}, \ \mbox{for all $\varphi\in \mathfrak{C}_c^{\infty}(\mathbb{R}^n)$},
\end{aligned}
\end{align}
 where in the above identity, $\widetilde{I}_{j}$ for $j=2,3,4,5$ are given by  
\begin{align*}
    \begin{aligned}
\widetilde{I}_{2}:=\int_{\Omega_T}
q(x)\left(\sum_{i=1}^{N}\frac{\beta_{j}(t,x)}{\tau^{i}}\right)
\varphi^{\ell}(x+t\omega)
\exp\!\left(-\frac{\ell}{2}\int_{0}^{t}a(x+\tau\omega)\,d\tau\right)\,dx\,dt,
\end{aligned} 
\end{align*}
\begin{align*}
\begin{aligned}
\widetilde{I}_{3}&:=\int_{\Omega_T}
q(x)\Bigg(
\varphi(x+t\omega)
\exp\!\left(\int_{0}^{t}\frac{a(x+\tau\omega)}{2}\,d\tau\right)
+\sum_{i=1}^{N}\frac{\beta_{i}(t,x)}{\tau^{i}}
\Bigg)\\
&\quad\times
\left(\sum_{j=1}^{\ell}
\varphi^{\ell-j}(x+t\omega)
\exp\!\left(-\frac{\ell-j}{2}\int_{0}^{t}a(x+\tau\omega)\,d\tau\right)
\left(\sum_{i=1}^{N}\frac{m_{i}(t,x)}{\tau^{i}}\right)^{j}
\right)\,dx\,dt, 
\end{aligned} 
\end{align*} 
\begin{align*}
\begin{aligned}
\widetilde{I}_{4}&:=\int_{\Omega_T}
q(x)\,e^{-i\ell\tau(t+x\cdot\omega)}
\Bigg(
\varphi(x+t\omega)
\exp\!\left(\int_{0}^{t}\frac{a(x+\tau\omega)}{2}\,d\tau\right)
+\sum_{i=1}^{N}\frac{\beta_{i}(t,x)}{\tau^{i}}
\Bigg)\\
&\ \times
\left(
\sum_{\emptyset\neq J\subset\{1,\dots,\ell\}}
\left(
e^{i\tau(t+x\cdot\omega)}
\left(
\varphi(x+t\omega)
\exp\!\left(-\int_{0}^{t}\frac{a(x+\tau\omega)}{2}\,d\tau\right)
+\sum_{i=1}^{N}\frac{m_{i}(t,x)}{\tau^{i}}
\right)
\right)^{\ell-|J|}
\prod_{j\in J}R_j
\right)\,dx\,dt 
\end{aligned} 
\end{align*} 
and 
\begin{align*}
\begin{aligned}
\widetilde{I}_{5}&:= \int_{\Omega_T}
q(x)\,R_0(t,x,\tau)
\Bigg(
\sum_{J\subset\{1,\dots,\ell\}}
\Bigg(
e^{i\tau(t+x\cdot\omega)}
\Bigg(
\varphi(x+t\omega)
\exp\!\left(-\int_{0}^{t}\frac{a(x+\tau\omega)}{2}\,d\tau\right)
\\&\qquad \qquad \qquad \qquad \qquad \qquad\qquad\qquad \qquad  +\sum_{i=1}^{N}\frac{m_{i}(t,x)}{\tau^{i}}
\Bigg)
\Bigg)^{\ell-|J|}
\prod_{j\in J}R_j
\Bigg)\,dx\,dt.
\end{aligned}
\end{align*} 
 The integrals $\widetilde{I_j}$, $j=2,3,4,5$, admit straightforward estimates based on the boundedness of continuous functions over compact sets together with the estimates~\eqref{eq: r_{(2)}} and \eqref{eq: r__{(2)}}. For the sake of completeness, we provide the estimates for each of these terms as follows. In the below mentioned estimates of $\widetilde{I}_j$,  $j=2,3,4,5$, the constant $C>0$ is independent of $\tau$. 

\vspace{.2cm}
\noindent\textbf{Estimate of $\widetilde I_2$.}
Each term in $\widetilde I_2$ contains at least one factor $\tau^{-1}$.
Since all remaining factors are uniformly bounded in $\Omega_T$, we obtain
\begin{align}
|\widetilde I_2|
\le C\,|\Omega_T|\,\tau^{-1}
\le \frac{C}{\tau}.  
\end{align}

\noindent\textbf{Estimate of $\widetilde I_3$.}
Similarly, every term in $\widetilde I_3$ involves powers of
$\tau^{-1}$ coming from the terms $m_j/\tau^j$.
Hence, we have
\begin{align}
|\widetilde I_3|
\le \frac{C}{\tau}.  
\end{align}

\noindent\textbf{Estimate of $\widetilde I_4$.}
In this case, each summand contains at least one factor $R_k$.
Using $\|R_k\|_{L^\infty}\le C/\tau$ and boundedness of the remaining
terms, we obtain
\begin{align}
    |\widetilde I_4|
\le C\,\|R_k\|_{L^\infty(\Omega_T)}
\le \frac{C}{\tau}.
\end{align}

\noindent\textbf{Estimate of $\widetilde I_5$.}
Here every term contains the remainder $R_0(t,x,\tau)$, and therefore we have
\begin{align}
    |\widetilde I_5|
\le C\,\|R_0\|_{L^\infty(\Omega_T)}
\le \frac{C}{\tau}.
\end{align}
In particular, we have seen that each $\widetilde{I_j}$ admits a bound of order $\mathcal{O}(\tau^{-1})$
, so  each $\widetilde{I_j}$ terms vanish as $\tau \to \infty$.
Thus, letting $\tau \to \infty$ in  \eqref{eq:final-splitting in q}, we arrive at the following identity 
\begin{align*}
\int_{\Omega_{T}}q(x)\varphi^{\ell+1}(x+t\omega)\left(A^{+}(t,x)\right)^{\ell-1}\,dx\,dt\,= \,\text{known},\ \mbox{for all $\varphi\in \mathfrak{C}_c^{\infty}(\mathbb{R}^n)$,} 
\end{align*}
where in the above step, $\omega\in\mathbb{S}^{n-1}$, is same as  in the expression of asymptotic solutions. 
Now since $q\in C_c^{\infty}(\Omega)$, we extend $q$ by $0$ outside of $\Omega$ (still denote by $q$), to  obtain 
\begin{align*}
\int_0^T\int_{\mathbb{R}^n} q(x)\varphi^{\ell+1}(x+t\omega)\left(A^{+}(t,x)\right)^{\ell-1}\,dx\,dt = \text{known},\ \mbox{for all $\varphi\in \mathfrak{C}_c^{\infty}(\mathbb{R}^n)$}.
\end{align*}
 Using the change of variable $x +t\omega=y$,  we have
\begin{align*}
\int_0^T\int_{\mathbb{R}^n} q(y -t\omega)\varphi^{\ell+1}(y)\left(A^{+}(t,y -t\omega)\right)^{\ell-1}\,dy\,dt = \text{known},\ \mbox{for all $\varphi\in \mathfrak{C}_c^{\infty}(\mathbb{R}^n)$}.
\end{align*}
Since the integrand is absolutely integrable (because $q$ has compact support), the  use of Fubini's Theorem yields that 
\begin{align}
\int_{\mathbb{R}^n}\left(\int_0^T q(y -t\omega)\left(A^{+}(t,y -t\omega)\right)^{\ell-1}\,dt\right)\varphi^{\ell+1}(y)\,dy = \text{known},\ \mbox{for all $\varphi\in \mathfrak{C}_c^{\infty}(\mathbb{R}^n)$ and $\omega\in\mathbb{S}^{n-1}$}.
\end{align}
Now, if we define 
$\Psi(y) := \int_0^T q(y -t\omega)\left(A^{+}(t,y -t\omega)\right)^{\ell-1}\,dt$, for $y\in \mathbb{R}^n$, then the above equation reduces to 
\begin{align}
\int_{\mathbb{R}^n}\Psi(y)\varphi^{\ell+1}(y)\,dy = \text{known},\ \mbox{for all $\varphi\in \mathfrak{C}_c^{\infty}(\mathbb{R}^n)$}.
\end{align}
 Now for $\chi\in C_{c}^{\infty}(B_1(0))$ such that  $0\leq \chi \leq 1$ and $\int_{\mathbb{R}^n}\chi^{\ell+1}(x)dx=1$ and  $y_0\in \Omega_{r}$  with $r \in \left(0, \min\left\{1, \frac{T - \operatorname{diam}(\Omega)}{3}\right\}\right)$ (see \eqref{dom:Omega_epsilon} for definition of  $\Omega_r$), we consider a  family of test functions $\{\chi_{\eta}\}_{\eta>0}\subset C_c^{\infty}(\R^n)$ by 
\begin{align}\label{eq;Phi_{delta}}
    \chi_{\eta}(y):=\eta^{\tfrac{-n}{{\ell+1}}}\chi\left(\frac{y-y_0}{\eta}\right),\ \quad \eta>0 \ \mbox{and}\ y\in\mathbb{R}^n.
\end{align}
 By the construction of $\{\chi_{\eta}\}_{\eta>0}$ and $\Omega_{r}$, we  observe that $\{\chi_{\eta}\}_{\eta>0}\subset \mathfrak{C}_c^{\infty}(\mathbb{R}^n)$, whenever $0<\eta<r$.
 Thus, we get that  
\begin{align}\label{eq; distribution = known with function q}
\int_{\mathbb{R}^n}\Psi(y)\chi_{\eta}^{\ell+1}(y)\,dy = \text{known},  \ \mbox{for all}\ r>\eta>0 \ \mbox{and} \ (y_0,\omega)\in\Omega_r\times \mathbb{S}^{n-1}. 
\end{align}
By a similar analysis as we did earlier, we deduce that the function $ \Psi$ is known in $\Omega_{r}$, where $\Omega_{r}$ is as defined in Equation \eqref{dom:Omega_epsilon}. Thus, for any $y_0 \in \Omega_{r}$, we have 
\begin{align*}
    \text{known} = \Psi(y_{0}) &:= \int_0^T q(y_{0} -t\omega)\left(A^{+}(t,y_{0} -t\omega)\right)^{\ell-1}\,dt \\
    & = \int_0^T q(y_{0} -t\omega)\exp\left(-\frac{(\ell-1)}{2}\int_0^t a(y_{0}-\rho\omega)\,d\rho\right)\,dt, \quad \text{where} \quad \rho = t-s.
\end{align*}
Now the fact that  $q$ is compactly supported in $\Omega$ and $(\Omega_{r} \pm K\omega)\cap \Omega = \emptyset, \text{ for all }K\geq T$,  gives us
$q(y_{0}-t\omega)=0$, whenever  $t\geq T$ and $y_0\in\Omega_r$. Hence,  we have
\begin{align*}
    \Psi(y_0) = \int_0^{\infty} q(y_{0} -t\omega)\exp\left(-\frac{(\ell-1)}{2}\int_0^ta(y_{0}-\rho\omega)\,d\rho\right)\,dt \, =\,  \text{known}, \  \text{ for } y_{0}\in \Omega_{r}. 
\end{align*}
From here, after repeating the arguments similar to the one used for reconstructing the damping coefficient $a$, we conclude that
\begin{align}\label{Psi is known in Rn}
    \Psi(y) = \int_0^{\infty} q(y -t\omega)\exp\left(-\frac{(\ell-1)}{2}\int_0^ta(y-\rho\omega)\,d\rho\right)\,dt \, =\,  \text{known}, \  \text{ for } y\in\mathbb{R}^{n}. 
\end{align}
From this, we show that  $q(y)$ is known for each $y\in\Omega$. 

We start by defining 
\begin{align}
    G(t,y) := q(y - t\omega) \ \mbox{and} \  H(t,y) := \exp\left(-\frac{(\ell-1)}{2}\int_0^t a(y - s\omega)\, ds\right),\, \mbox{for  $(t,y)\in (0,\infty)\times \mathbb{R}^n$,}
\end{align}
so that the function $\Psi(y)$ becomes
\begin{align}\label{eq:int_1}
\Psi(y):= \int_0^{\infty} G(t,y) H(t,y)\,dt,\ \ \mbox{for}\ y\in \mathbb{R}^n.
\end{align}
Next, for $\gamma\in \mathbb{R}$, we define
\begin{align*}
    U_{\gamma}(y) := \Psi(y-\gamma\omega), \, \mbox{where $y\in \mathbb{R}^n$}.
\end{align*}
Differentiating $U_{\gamma}$ with respect to $\gamma$ and evaluating at $\gamma = 0$, we obtain the following identity
\begin{align}\label{eq: derivative of U_{gamma}}
    -\omega \cdot \nabla_y \Psi(y) = \left.\frac{d}{d\gamma} U_{\gamma}(y)\right|_{\gamma=0}
= \int_0^{\infty} \left.\frac{d}{d\gamma} \left[ G_\gamma(t,y) H_\gamma(t,y) \right] \right|_{\gamma=0} dt,
\end{align}
where
\begin{align}
    G_\gamma(t,y) := q(y - (\gamma + t)\omega), 
    \ \mbox{and}\  H_\gamma(t,y) := \exp\left(-\frac{(\ell-1)}{2}\int_0^t a(y - (\gamma + s)\omega)\, ds\right).
\end{align}
To solve the right-hand side of the Equation \eqref{eq: derivative of U_{gamma}}, first we differentiate $G_\gamma$ and $H_\gamma$ with respect to $\gamma$ and evaluate at $\gamma = 0$. Thus, using the following identity
\begin{align}
    \left.\frac{d}{d\gamma} \int_0^t a(y - (\gamma + s)\omega)\, ds \right|_{\gamma=0}
= a(y- t\omega) - a(y),
\end{align}
we obtain
\begin{align}\label{eq: derivative of H_{gamma}}
    \left.\frac{d}{d\gamma} H_\gamma(t,y) \right|_{\gamma=0}
= -\frac{(\ell-1)}{2}H(t,y)\left(a(y- t\omega) - a(y)\right),
\end{align}
and
\begin{align}\label{eq: derivative of G_{gamma}}
    \left.\frac{d}{d\gamma} G_\gamma(t,y) \right|_{\gamma=0}
= -\omega \cdot \nabla_y q(y - t\omega)
= \frac{d}{dt} q(y - t\omega) = G'(t,y).
\end{align}
After using the above Equations \eqref{eq: derivative of H_{gamma}} and \eqref{eq: derivative of G_{gamma}} in Equation \eqref{eq: derivative of U_{gamma}}, we obtain the following identity
\begin{align}\label{eq:int_2}
\hspace{-1cm}-\omega \cdot \nabla_y \Psi(y)
= \int_0^{\infty} \left[ G'(t,y) H(t,y) - \dfrac{(\ell-1)}{2}G(t,y)\left\{a(y- t\omega) - a(y)\right\} H(t,y) \right]\,dt.
\end{align}
Next, integrating $G'(t,y) H(t,y)$ by parts and using $H'(t,y) = -\frac{(\ell-1)}{2}a(y - t\omega) H(t,y)$, we get
\begin{align}\label{eq; apply IBP}
    \int_0^{\infty} G'(t,y) H(t,y)\, dt
&= \left[G(t,y) H(t,y)\right]_{t = 0}^{t = \infty} - \int_0^{\infty} G(t,y) H'(t,y)\, dt\nonumber\\
&= -q(y) + \frac{(\ell-1)}{2}\int_0^{\infty} G(t,y) a(y - t\omega) H(t,y)\, dt.
\end{align}
Combining \eqref{eq; apply IBP} and \eqref{eq:int_2} together with \eqref{eq:int_1}, we obtain the following transport Equation
\begin{align}
-\omega \cdot \nabla_{y}\, \Psi(y) = -q(y) +  \frac{(\ell-1)}{2}\,a(y)\, \Psi(y),\, \mbox{for all $y\in \mathbb{R}^n$}.
\end{align}
Equivalently, we have
\begin{align}\label{eq:transport-eq}
\omega \cdot \nabla_{y}\, \Psi(y) + \frac{(\ell-1)}{2} \,a(y)\, \Psi(y) = q(y),\, \mbox{for all $y\in \mathbb{R}^n$}.
\end{align}
Thus, finally using \eqref{Psi is known in Rn}, together with the fact that $q$ is zero outside $\Omega$, we conclude that  $q(y)$ is known for all $y\in\Omega.$
This completes the proof of the main Theorem \ref{th: main result}.\qed
    \section*{Acknowledgments}
\begin{itemize}
      \item R.~Bhardwaj gratefully acknowledges the Senior Research Fellowship from the UGC, the Government of India.
	\item M.~Kumar acknowledges the support of PMRF (Prime Minister's Research Fellowship) from the government of India for his research.
	\item M.~Vashisth work was supported by the ISIRD project 9--551/2023/IITRPR/10229 from IIT Ropar.
    \item This work was partially supported by the FIST program of the Department of Science and Technology, Government of India, Reference No. SR/FST/MS-I/2018/22(C).\\
\end{itemize}

     \noindent\textbf{Data availability statement.} \ Data sharing is not applicable to this article, as no datasets were generated or analyzed during the current study.\\
    
 \noindent\textbf{Conflict of interest.} \
The authors declared that they have no potential conflicts of interest with respect to the research, authorship, and/or publication of this article.
	
	 \bibliography{math} 
	
	 \bibliographystyle{alpha}

\end{document}